\documentclass[11pt]{amsart}
\usepackage{amsmath,amscd,amsthm,amssymb}

\usepackage[bookmarks]{hyperref}
\usepackage[all,cmtip]{xy}
\usepackage{verbatim}
\newtheorem{thm}{Theorem}[subsection]
\newtheorem{prop}[thm]{Proposition}

\newtheorem{cor}[thm]{Corollary}
\newtheorem{lem}[thm]{Lemma}
 
\newtheorem{assump}[thm]{Assumption} 

\theoremstyle{definition}

\newtheorem{rem}[thm]{Remark}

\numberwithin{equation}{subsection}

\newcommand{\ul}{\mathfrak{u}}

\newcommand{\Lie}{\operatorname{Lie}}

\newcommand{\ind}{\operatorname{ind}}

\newcommand{\opH}{\operatorname{H}}

\newcommand{\Hom}{\operatorname{Hom}}

\newcommand{\bw}{\bigwedge}
\newcommand{\lth}{\prec_{\operatorname{ht}}}

\newcommand{\ga}{\gamma}

\newcommand{\la}{\lambda}

\newcommand{\al}{\alpha}
\newcommand{\be}{\beta}

\newcommand{\si}{\sigma}

\newcommand{\de}{\delta}

\begin{document}

\title[Third Cohomology for Frobenius kernels and related structures]
{Third Cohomology for Frobenius kernels and related structures}
\author{\sc Christopher P. Bendel}
\address
{Department of Mathematics, Statistics and Computer Science\\
University of
Wisconsin-Stout \\
Menomonie\\ WI~54751, USA}
\thanks{Research of the first author was supported in part by a Simons Foundation Collaboration Grant {\#}317062}
\email{bendelc@uwstout.edu}
\author{\sc Daniel K. Nakano}
\address
{Department of Mathematics\\ University of Georgia \\
Athens\\ GA~30602, USA}
\thanks{Research of the second author was supported in part by NSF
grant DMS-1402271}
\email{nakano@math.uga.edu}
\author{\sc Cornelius Pillen}
\address{Department of Mathematics and Statistics \\ University of South
Alabama\\
Mobile\\ AL~36688, USA}
\thanks{Research of the third author was supported in part by a Simons Foundation Collaboration Grant {\#}245236}
\email{pillen@southalabama.edu}

\date\today

\begin{abstract} Let $G$ be a simple simply connected group scheme defined over ${\mathbb F}_{p}$ and $k$ be an algebraically closed field of characteristic $p>0$. Moreover, 
let $B$ be a Borel subgroup of $G$ and $U$ be the unipotent radical of $B$. In this paper the authors compute the third cohomology group for $B$ 
and its Frobenius kernels, $B_{r}$, with coefficients in a one-dimensional representation. These computations hold with relatively mild restrictions on the characteristic of the 
field.  As a consequence of our calculations,  the third ordinary Lie algebra cohomology group for ${\mathfrak u}=\text{Lie }U$ with coefficients in $k$ is determined, as well as 
the third $G_{r}$-cohomology with coefficients in the induced modules $H^{0}(\lambda)$. 
\end{abstract}

\maketitle


\section{Introduction}

\subsection{} Over the last 40 years, one of the central questions in the representation 
theory of algebraic groups that remains open is to understand the structure 
and vanishing behavior of the line bundle cohomology $H^{n}(\lambda):={\mathcal H}^{n}(G/B,{\mathcal L}(\lambda))$ for $n\geq 0$ where 
$G/B$ is the flag variety and ${\mathcal L}(\lambda)$ is a line bundle over $G/B$. 
For $n=0$, the characters are given by Weyl's character formula, but the composition factors are 
not known when $k={\overline {\mathbb F}}_{p}$. Andersen computed the socle of 
$H^{1}(\lambda)$, but more general results other than calculations for low rank examples remain elusive. 
A fundamental computation which is related to understanding line bundle cohomology is the calculation of 
rational $B$-cohomology, $\text{H}^{\bullet}(B,\lambda)$, where $B$ is a Borel subgroup for the reductive group $G$ and
$\lambda\in X(T)$ is a weight regarded as a one-dimensional representation of $B$. 

In \cite{BNP2} the authors calculated $\text{H}^{2}(B,\lambda)$ for $p\geq 3$. Our methods involved weaving other relevant 
$\text{H}^{2}$-calculations for Frobenius kernels and Lie algebras into the picture. These fundamental 
cohomology groups and the status of their computation are presented in the following table.\footnote{See the following section for the notation.} The references are 
\cite{Kos}, \cite{FP2}, \cite{PT}, \cite{AJ},  \cite{KLT}, \cite{Jan2}, \cite{BNP1, BNP2}, \cite{AR}, \cite{W}, \cite{UGA}. 

\begin{center}
\renewcommand{\arraystretch}{1.3}
\begin{tabular}{ccc}
& Cohomology Group    & Known Results  \\ \hline  
(1) &  $\text{H}^{n}({\mathfrak u},k)$ & [$p\geq h-1$, $n\geq 0$]; [$p\geq 2$, $n=0,1$]; [$p\geq 3$. $n=2$]       \\
(2) &  $\text{H}^{n}(U_{1},k)$  & [$p\geq h$, $n\geq 0$]; [$p\geq 2$, $n=0,1,2$]   \\
(3) &  $\text{H}^{n}(B_{1},\lambda)$ & [$p\geq h$, $n\geq 0$]; [$p\geq 2$, $n=0,1,2$]   \\
(4) &  $\text{H}^{n}(B_{r},\lambda)$ & [$p\geq 2$, $n=0,1,2$]   \\                          
(5) &  $\text{H}^{n}(B,\lambda)$ & [$p\geq 2$, $n=0,1,2$]; [$p>h$, $n=3$]   \\    
(6) &  $\text{H}^{n}(G_{r},H^{0}(\lambda))$ & [$p>h$, $r=1$, $n\geq 0$]; 
[$p\geq 2$, $r\geq 1$, $n=0,1,2$]   \\ \hline
\end{tabular}
\end{center} 
Recently, for $SL_2$, Ngo \cite{Ngo} has computed (4) and (6) for all primes and all $n$.  

The calculation of the aforementioned cohomology groups is 
formidable and of general interest. 
For example, the complete calculation of (1) and (2) would yield  
a general version (for {\em all} characteristics) of the celebrated theorem of Kostant \cite{Kos}. 
For (6) (when $\lambda=0$), the computation of $\text{H}^{\bullet}(G_r,k)$ for  
$r\geq 2$ presents a major challenge and is geometrically related to the variety 
of commuting $r$-tuples of $p$-nilpotent matrices \cite{SFB1,SFB2}.  
The goal in this paper is to expand the computations of (1)-(6) to the case when $n=3$. 

The paper is organized as follows. In Section 2 we provide a preliminary analysis on possible weights that can 
occur in the ordinary Lie algebra cohomology group $\text{H}^{3}({\mathfrak u},k)$. 
At the beginning of Section 3,  
further results are provided on constraints involving root sums. These results are crucial throughout the paper in 
considering differentials in various spectral sequences. Later in this section, we provide a realization of $\text{H}^{3}(U_{1},k)$ via 
the ordinary cohomology groups $\text{H}^{j}({\mathfrak u},k)$ for $j=1,3$. 

At the beginning of Section 4, we introduce assumptions on the characteristic of the field that will be used during the remainder of the paper. Several key ideas 
from Andersen \cite{And} involving the $B$-cohomology are then employed to give an explicit description of $\text{H}^{3}({\mathfrak u},k)$. The section concludes by establishing 
another crucial calculation of ordinary Lie algebra cohomology, namely $\text{H}^{1}({\mathfrak u},{\mathfrak u}^{*})$. 

In Section 5, with these prior computations and some arguments using spectral sequences, we determine $\text{H}^{3}(B_{r},\lambda)$ for $\lambda\in X(T)$. As 
a consequence of these results, the $B$-cohomology groups, $\text{H}^{n}(B,\lambda)$, for $\lambda\in X(T)$ are determined which extends the prior work of 
Andersen and Rian \cite{AR}. Finally, in Section 6, we give a description of $\text{H}^{3}(G_{r},H^{0}(\lambda))^{(-r)}$ and demonstrate that these cohomology 
groups (as $G$-modules) admit a good filtration. An application is provided at the end of the section, which uses the $B$-cohomology calculations to give linear 
bounds on the third cohomology of finite Chevalley groups. An Appendix at the end of hte paper contains the computations of certain weights that appear in the calculations of $B_1$-cohomology groups.

\subsection{\bf Notation:} Throughout this paper, we will
follow the basic conventions provided in \cite{Jan1}.
\vskip .5cm 
\begin{enumerate}
\item[(1)] $k$: an algebraically closed field of characteristic $p>0$.

\item[(2)] $G$: a simple, simply connected algebraic group which is defined 
and split over the finite prime field ${\mathbb F}_p$ of characteristic $p$. The assumption
that $G$ is simple (equivalently, its root system $\Phi$ is irreducible) is largely one of convenience.
All the results of this paper extend easily to the semisimple, simply connected case.

\item[(3)] $F:G\rightarrow G$: the Frobenius morphism. 

\item[(4)] $G_r=\text{ker }F^{r}$: the $r$th Frobenius kernel of $G$. 

\item[(5)] $G^{(r)}$: the $r$th Frobenius twist of $G$; $G^{(r)} \cong G/G_r$.

\item[(6)] $G({\mathbb F}_{q})$: the associated finite Chevalley group where ${\mathbb F}_q$ is the field with $q = p^r$ elements. 

\item[(7)] $T$: a maximal split torus in $G$. 

\item[(8)] $\Phi$: the corresponding (irreducible) root system associated to $(G,T)$. When referring to short and long roots, when a root system has roots of only one length, all roots shall be considered as both short and long.

\item[(9)] $\Pi=\{\alpha_1,\dots,\alpha_n\}$: the set of simple roots. We will adhere to the ordering of the simple roots as given in
\cite{Jan2} (following Bourbaki). In particular, for type $B_n$, $\al_n$ denotes the unique short simple root and for type $C_n$,
$\al_n$ denotes the unique long simple root.   

\item[(10)] $\Phi^{\pm}$: the positive (respectively, negative) roots.  

\item[(11)] $\alpha_0$: the maximal short root. 

\item[(12)]  $B$: a Borel subgroup containing $T$ corresponding to the negative roots. 

\item[(13)] $U$: the unipotent radical of $B$.

\item[(14)] $\mathbb E$: the Euclidean space spanned by $\Phi$ with inner product $\langle\,,\,\rangle$ normalized so that $\langle\alpha,\alpha\rangle=2$ for $\alpha \in \Phi$ any short root.

\item[(15)] $\alpha^\vee=2\alpha/\langle\alpha,\alpha\rangle$: the coroot of $\alpha\in \Phi$.

\item[(16)] $\rho$: the Weyl weight defined by $\rho=\frac{1}{2}\sum_{\alpha\in\Phi^+}\alpha$.

\item[(17)] $h$: the Coxeter number of $\Phi$, given by $h=\langle\rho,\alpha_0^{\vee} \rangle+1$.

\item[(18)] $W=\langle s_{\alpha_1},\dots,s_{\alpha_n}\rangle\subset{\mathbb O}({\mathbb E})$: the Weyl group of $\Phi$, generated by the orthogonal reflections $s_{\alpha_i}$, $1\leq i\leq n$. For $\alpha\in \Phi$, $s_\alpha:{\mathbb E} \to{\mathbb E}$ is the orthogonal reflection in the hyperplane $H_\alpha\subset \mathbb E$ of vectors orthogonal to $\alpha$.

\item[(19)] $\ell : W \to {\mathbb Z}$: the usual length function on $W$; for $w \in w$, $\ell(w)$ is the minimum number of simple reflections required to express $w$ as a product of simple reflections.

\item[(20)] $X(T)=\mathbb Z \omega_1\oplus\cdots\oplus{\mathbb Z}\omega_n$: the weight lattice, where the fundamental dominant weights $\omega_i\in{\mathbb E}$ are defined by $\langle\omega_i,\alpha_j^\vee\rangle=\delta_{ij}$, $1\leq i,j\leq n$.

\item[(21)] $X(T)_{+}={\mathbb N}\omega_1+\cdots+{\mathbb N}\omega_n$: the cone of dominant weights.

\item[(22)] $X_{r}(T)=\{\lambda\in X(T)_+: 0\leq \langle\lambda,\alpha^\vee\rangle<p^{r},\,\,\forall \alpha\in\Pi\}$: the set of $p^{r}$-restricted dominant weights. 


\item[(23)] $M^{(s)}$:  the module obtained by composing the underlying representation for 
a rational $G$-module $M$ with $F^{s}$.

\item[(24)] $H^0(\lambda) := \operatorname{ind}_B^G\lambda$, $\lambda\in X(T)_{+}$: the induced module whose character is provided by Weyl's character formula.  


\end{enumerate}


\section{Observations on $\ul$-cohomology}


\subsection{}  We begin by recalling the definition of the ordinary Lie algebra
cohomology of a Lie algebra $L$ over $k$. The ordinary
Lie algebra cohomology $\opH^i(L,k)$ may be computed as the
cohomology of the complex
$$
k \overset{d_0}\to L^* \overset{d_1}\to \Lambda^2(L)^*
\overset{d_2}\to \Lambda^3(L)^* \to \cdots .
$$
The differentials are given as follows:
$d_0 = 0$ and $d_1: L^* \to \Lambda^2(L)^*$ with
\begin{eqnarray*}
(d_1\phi)(x\wedge y) = -\phi([x,y])
\end{eqnarray*}
where $\phi \in L^*$ and $x, y \in L$. For the
higher differentials, we identify $\Lambda^n(L)^* \cong \Lambda^n(L^*)$.
Then the differentials are determined by the following
product rule (see \cite[I.9.17]{Jan1}):
\begin{eqnarray*}
d_{i+j}(\phi\wedge\psi) = d_i(\phi)\wedge\psi + (-1)^i\phi\wedge
d_j(\psi).
\end{eqnarray*}
Using this formula, one can use induction to obtain the following formula.

\begin{lem}\label{L:differential} Let $n \geq 1$ be an integer. Consider $d_n : \Lambda^n(L^*) \to \Lambda^{n+1}(L^*)$.
Let 
$$
x = \bw_{i = 1}^{n} \phi_i \in \Lambda^n(L^*).
$$
Then
$$
d_n(x) = \sum_{j = 1}^{n}(-1)^{j + 1}\bw_{i=1}^n\phi_{i}^j
$$
where $\phi_i^j = 
\begin{cases}
\phi_i &\text{ if } i \neq j\\
d_1(\phi_i) & \text{ if } i = j.
\end{cases}$

In other words,
\begin{align*}
d_n(x) &= d_1(\phi_1)\wedge\phi_2\wedge\phi_3\wedge\cdots\wedge\phi_n -
	\phi_1\wedge d_1(\phi_2) \wedge \phi_3 \wedge \cdots\wedge\phi_n\\
	&\qquad + \cdots + 
	(-1)^{n+1}\phi_1\wedge\phi_2\wedge\cdots\wedge\phi_{n-1}\wedge d_1(\phi_n).
\end{align*}
\end{lem}

For the remainder of the paper we will be primarily interested in the ordinary Lie algebra cohomology
of the Lie algebra $\ul = \Lie(U)$ where $U$ is the unipotent radical of a Borel subgroup of $G$.


\subsection{\bf First cohomology of $\ul$:}
Now let us consider the case when $L={\mathfrak u}=\text{Lie }U$.
A basis for ${\mathfrak u}$ is given by a basis of negative root vectors
$\{x_{\alpha}: \alpha\in \Phi^{-}\}$. Let $\{\phi_{\alpha}:\ \alpha\in
\Phi^{+}\}$ be the dual basis in ${\mathfrak u}^{*}$ with
$\phi_{\alpha}(x_{\beta})=\delta_{-\alpha,\beta}$ for all
$\alpha\in \Phi^{+}$ and $\beta\in \Phi^{-}$.
It is well known that the first cohomology (for any Lie algebra)
is
$$
\opH^1(\ul,k) = \ker d_1 = (\ul/[\ul,\ul])^*.
$$
For large primes, the simple roots give a basis for
$\opH^1(\ul,k)$. Specifically, we recall the following result of
Jantzen \cite{Jan2}.

\begin{prop}\label{P:h1} Assume $p \geq 3$.
\begin{itemize}
\item[(a)] If $p = 3$, assume further that $\Phi$ is not of type $G_2$.
Then a $T$-basis for $\opH^1(\ul,k)$ is $\{\phi_{\al}~:~ \al \in \Pi\}$.
\item[(b)] If $p = 3$ and $\Phi$ is of type $G_2$, then
a $T$-basis for $\opH^1(\ul,k)$ is $\{\phi_{\al_1},\phi_{\al_2},
\phi_{3\al_1 + \al_2}\}$.
\end{itemize}
\end{prop}


\subsection{\bf Second cohomology of $\ul$:} For $p \geq 3$, the second cohomology
groups $\opH^2(\ul,k)$ were computed by the authors in \cite[Thm. 4.4]{BNP2}. We remind the
reader of the results here.

\begin{thm}\label{T:h2} Let $p \geq 3$ and $\pi = \{ -w\cdot 0 ~ : ~ w \in W,\ l(w) = 2\}$.
As a $T$-module
$$
\opH^2(\ul,k) \cong \bigoplus_{\la \in \pi \cup \pi'}\la
$$
where $\pi'$ is given below.  
Moreover, if $\la = -w\cdot 0$ with $w = s_{\al}s_{\be}$, then the 
corresponding cohomology class
is represented by 
$\phi_{\al}\wedge\phi_{-\langle\be,\al^{\vee}\rangle\al + \be}$.
\begin{itemize}
\item[(a)] $p > 3$ or $\Phi$ is of type $A_n$, $B_2 = C_2$, $D_n$, or $E_n$: 
$\pi' = \emptyset$.
\item[(b)] $p = 3$ and $\Phi$ is of type $B_n$, $n \geq 3$: 
$\pi' = \{\al_{n-2} + 2\al_{n-1} + 3\al_n\}$ corresponding to the
cohomology class
$$
\phi_{\al_{n}}\wedge\phi_{\al_{n-2} + 2\al_{n-1} + 2\al_{n}}
- \phi_{\al_{n-1} + \al_{n}}\wedge\phi_{\al_{n-2} + \al_{n-1} + 2\al_{n}}
+ \phi_{\al_{n-2} + \al_{n-1} + \al_{n}}\wedge\phi_{\al_{n-1} + 2\al_{n}}.
$$
\item[(c)] $p = 3$ and $\Phi$ is of type $C_n$, $n \geq 3$: 
$\pi' = \{\al_{n-2} + 3\al_{n-1} + \al_n\}$ corresponding to the 
cohomology class
$$
\phi_{\al_{n-1}}\wedge\phi_{\al_{n-2}+2\al_{n-1}+\al_n} -
\phi_{\al_{n-2} + \al_{n-1}}\wedge\phi_{2\al_{n-1} + \al_{n}}.
$$
\item[(d)] $p = 3$ and $\Phi$ is of type $F_4$: 
$\pi' = \{\al_1 + 2\al_2 + 3\al_3, \al_{2} + 3\al_{3} + \al_4\}$
corresponding to the cohomology classes
$$
\phi_{\al_3}\wedge\phi_{\al_1 + 2\al_2 + 2\al_3} 
- \phi_{\al_2 + \al_3}\wedge\phi_{\al_1 + \al_2 + 2\al_3}
+ \phi_{\al_1 + \al_2 + \al_3}\wedge\phi_{\al_2 + 2\al_3},
$$
$$
\phi_{\al_3}\wedge\phi_{\al_2+2\al_3+\al_4} -
\phi_{\al_3 + \al_4}\wedge\phi_{\al_2 + 2\al_3}.
$$
\item[(e)] $p = 3$ and $\Phi$ is of type $G_2$: $\pi' = \{3\al_1 + \al_2,
3\al_1 + 3\al_2, 6\al_1 + 3\al_2, 4\al_1 + 2\al_2\}$
corresponding to the cohomology classes
$$
\phi_{\al_1}\wedge\phi_{2\al_1 + \al_2},
\phi_{\al_2}\wedge\phi_{3\al_1 + 2\al_2},
\phi_{3\al_1 + \al_2}\wedge\phi_{3\al_1 + 2\al_2},$$
$$\phi_{\al_1}\wedge\phi_{3\al_1 + 2\al_2} + 
        \phi_{\al_1 + \al_2}\wedge\phi_{3\al_1 + \al_2}.
$$
\end{itemize}
\end{thm}


\subsection{\bf General results:} We present a general observation about weight spaces of $\Lambda^n(\ul^*)$.

\begin{prop}(\cite[Prop. 2.2]{FP2}, \cite[Prop. 2.3]{BNP2}) Let $w \in W$. Then
$$
\dim_k \Lambda^n(\ul^*)_{-w\cdot 0} =
\begin{cases}
1 &\text{ if } n = l(w)\\
0 &\text{ otherwise}.
\end{cases}
$$
Let $x_w \in \Lambda^{l(w)}(\ul^*)$ be an element of weight $-w\cdot 0$.
Then $x_w$ represents a cohomology class in $\opH^{l(w)}(\ul,k)$.
\end{prop}

Over characteristic zero, $\opH^{\bullet}(\ul,k)$ was computed by
Kostant \cite{Kos}. The cohomology classes in the preceding
proposition in fact yield a $T$-basis. In prime characteristic, it is known for
$p \geq h - 1$ by work in  \cite{FP2} \cite{PT} and \cite{UGA} that the formal characters of these
cohomology groups are the same as in characteristic zero. Our goal will be 
to show that this holds for $\opH^3(\ul,k)$ when the prime $p$ is not too small. 
See Theorem \ref{T:u-coho}.


\subsection{} We first investigate the nature of weights $-w \cdot 0$.  Let $w \in W$ have $\ell(w) = m$.  
Then $w$ can be expressed in reduced form as $w = s_1s_2\cdots s_m$ where $s_i = s_{\be}$ 
for some $\be \in \Pi$.  Inductively, one obtains the following.

\begin{prop}\label{P:wdotzero} Given $w = s_1s_2\cdots s_m$, let $\be_i$ denote the simple root
corresponding to the simple reflection $s_i$.  Then
\begin{align*}
-w\cdot 0 &= - s_1s_2\cdots s_m(\be_m) - s_1s_2\cdots s_{m-1}(\be_{m-1}) - \dots
			- s_1s_2(\be_2) - s_1(\be_1)\\
	& = - s_1s_2\cdots s_{m-1}(-\be_m) - s_1s_2\cdots s_{m-2}(-\be_{m-1}) - \dots
			- s_1(-\be_2) - (- \be_1)\\
	&= s_1s_2\cdots s_{m-1}(\be_m) + s_1s_2\cdots s_{m-2}(\be_{m-1}) + \dots
			+ s_1(\be_2) + \be_1.
\end{align*}
Furthermore, each of the summands lies in $\Phi^{+}$.  
\end{prop}

In particular, suppose that $w \in W$ has length 3.  
Write $w = s_{\be_1}s_{\be_2}s_{\be_3}$.  Then
\begin{align*}
-w\cdot 0 &= -s_{\be_1}s_{\be_2}s_{\be_3}\cdot 0\\
	& = s_{\be_1}s_{\be_2}(\be_3) + s_{\be_1}(\be_2) + \be_1\\
	&= [\be_3 - \langle\be_3,\be_2^{\vee}\rangle\be_2 - 
			(\langle\be_3,\be_1^{\vee}\rangle - \langle\be_3,\be_2^{\vee}\rangle\langle\be_2,\be_1^{\vee}\rangle)\be_1]
			+ [\be_2 - \langle\be_2,\be_1^{\vee}\rangle\be_1] + [\be_1]
\end{align*}
where the roots in brackets are each positive roots.

One can now  conclude the following.

\begin{cor}\label{C:h3wdot0} Let $\be_1, \be_2, \be_3 \in \Pi$ such that
$w = s_{\be_1}s_{\be_2}s_{\be_3}$ has length 3.
Then $\phi_{\be_1}\wedge\phi_{s_{\be_1}(\be_2)}\wedge\phi_{s_{\be_1}s_{\be_2}(\be_3)}$
has weight $-w\cdot 0$  and
represents a cohomology class in $\opH^3(\ul,k)$.
\end{cor}

Note for $w = s_{\be_1}s_{\be_2}s_{\be_3} \in W$ with $\ell(w) = 3$, $-w\cdot 0$ involves at most 
three simple roots.  Indeed, if $\be_1$, $\be_2$, and $\be_3$ are all distinct, then 
$-w\cdot 0 = i\be_1 + j\be_2 + \be_3$ for some $i, j > 0$ (with a more precise formula
determined as above). On the other hand, suppose the
simple roots are not distinct.  In order to be of length three, the only possibility is
that $\be_1 = \be_3$ and $\be_2$ must be adjacent to $\be_1$. Then we have that
$-w\cdot 0 = -s_{\be_1}s_{\be_2}s_{\be_1}\cdot 0 = i\be_1 + j\be_2$ for some $i \geq j > 0$.


\subsection{} We now identify some limitations on which other
wedge products $\phi_{\al}\wedge\phi_{\be}\wedge\phi_{\ga}$ or linear
combinations thereof can represent cohomology classes.
Since the differentials are additive and preserve the action of
$T$, of interest are linear combinations of wedge products that
have the same weight. To avoid ``trivial'' linear combinations,
we say that an expression
$\sum c_{\al,\be,\ga}\phi_{\al}\wedge\phi_{\be}\wedge\phi_{\ga} \in \Lambda^3(\ul^*)$
is in {\it reduced} form if 
a triple $(\al,\be,\ga)$ appears at most once and each $c_{\al,\be,\ga} \neq 0$.

While we are interested particularly in degree 3, me make the following general 
observation which extends \cite[Prop. 2.4]{BNP2}.

\begin{prop}\label{P:d1} Let $x = \sum_j c_j\bw_{i = 1}^{n}\phi_{\si_{i,j}}$
be an element in $\Lambda^n(\ul^*)$ in reduced form of weight $\gamma$
for some $\gamma \in X(T)$. If $d_n(x) = 0$, then $d_1(\phi_{\si_{i,j}}) = 0$
for at least one $\si_{i,j}$ appearing in the sum. 
\end{prop}

\begin{proof} Fix an ordering on $\Phi^+$ such that $\Phi^+$ may be
identified as $\Phi^+ = \{\ga_i ~ : ~ 1 \leq i \leq |\Phi^+|\}$ 
with $\text{ht}(\ga_i) \leq \text{ht}(\ga_{i+1})$ for all $i$. 
By reordering if necessary, we may assume that in 
each wedge $\bw_{i = 1}^{n}\phi_{\si_{i,j}}$, we have
$\si_{i,j} \lth \si_{i+1,j}$.  In particular, $\text{ht}(\si_{i,j}) \leq \text{ht}(\si_{i+1,j})$.

Using the ordering given by $\lth$, order the wedge products found in $x$ 
lexicographically based on $\bw_{i = 2}^{n}\phi_{\si_{i,j}}$.  That is, we ignore
the first element in the wedge when forming this ordering.  Note that, for each $j$, since
$\ga = \sum_{i = 1}^{n}\si_{i,j}$, $\si_{1,j}$ is determined if one knows $\ga$ and 
$\{\si_{i,j} ~:~ 2 \leq i \leq n\}$.  In this lexicographical ordering, choose the 
wedge product that is maximal.  For simplicity, we denote this simply by
$x_{max} = \bw_{i = 1}^n\phi_{\si_i}$.  By our construction, for all $j$, we have
$$
\text{ht}\left(\bw_{i=2}^n\phi_{\si_i}\right) := \sum_{i=2}^n\text{ht}(\si_i) \geq 
	\sum_{i=2}^n\text{ht}(\si_{i,j}) = \text{ht}\left(\bw_{i = 2}^{n}\phi_{\si_{i,j}}\right).
$$
Note that equality is possible.  Although in such a case, the actual roots must differ
in some manner since $x$ is assumed to be in reduced form.

We claim that in this element $x_{max}$, we must have $d_1(\phi_{\si_1}) = 0$. 
By assumption, $d_n(x) = 0$ and the coefficient of $x_{max}$ is non-trivial.  From Lemma \ref{L:differential}, 
we have
$$
d_n(x_{max}) = d_1(\phi_{\si_1})\wedge\bw_{i = 2}^{n}\phi_{\si_i} + 
	\sum_{k = 2}^{n}(-1)^{k + 1}\bw_{i=1}^{n}\phi_{\si_i}^k.
$$
For each other summand of $x$, we get a similar expression for the result of 
applying $d_n$.  

Consider the term $d_1(\phi_{\si_1})\wedge\bw_{i = 2}^{n}\phi_{\si_i}$.  Either
this is zero or it must cancel with another term of the form
\begin{equation}\label{Eq:wedge1}
\phi_{\si_{1,j}}\wedge\cdots\wedge\phi_{\si_{l-1,j}}\wedge d_1(\phi_{\si_{l,j}})\wedge
	\phi_{\si_{l+1,j}}\wedge\cdots\wedge\phi_{\si_{n,j}}
\end{equation}
for some integers $j,l$.  We show that the latter cannot happen.

Suppose on the contrary that it did. Then, for each $2 \leq i \leq n$, 
$\phi_{\si_i}$ would have to appear within a wedge as in (\ref{Eq:wedge1}). This 
could happen by involving zero, one, or two terms from the double wedge $d_1(\phi_{\si_{l,j}})$.  

If it involved no components of $d_1(\phi_{\si_{l,j}})$, this would mean that 
$$
\bw_{i = 2}^{n}\phi_{\si_i} = \pm\bw_{i \neq l}\phi_{\si_{i,j}}.
$$
As noted above, by weight considerations, this would imply that $\si_1 = \si_{l,j}$ and moreover
that 
$$
\bw_{i = 1}^{n}\phi_{\si_i} = \pm\bw_{i = 1}^{n}\phi_{\si_{i,j}}
$$
which contradicts the fact that $x$ is in reduced form.

For the remaining cases, note that for any $\eta \in \Phi^{+}$, if
$d_1(\phi_{\eta}) = \sum c_{\al,\be}\phi_{\al}\wedge\phi_{\be}$, then
$\text{ht}(\al) < \text{ht}(\eta)$ and $\text{ht}(\be) < \text{ht}(\eta)$.

The second case would be that $\bw_{i = 2}^{n}\phi_{\si_i}$ is (up to a sign)
the wedge of all but one of the $\phi_{\si_{i,j}}$ (with $i \neq l$) along with 
a $\phi_{\al}$ appearing in $d_1(\phi_{\si_{l,j}})$.  Since
$\text{ht}(\al) < \text{ht}(\si_{l,j})$ and $\text{ht}(\si_{1,j}) \leq \text{ht}(\si_{i,j})$ for $2 \leq i \leq n$, 
$$
\text{ht}\left(\bw_{i=2}^n\phi_{\si_i}\right) <
\text{ht}\left(\bw_{i = 2}^{n}\phi_{\si_{i,j}}\right) \leq
\text{ht}\left(\bw_{i=2}^n\phi_{\si_i}\right).
$$
We can conclude that this is not possible.

Similarly, in the third case, we would have that $\bw_{i = 2}^{n}\phi_{\si_i}$ is (up to a sign)
the wedge of all but two of the $\phi_{\si_{i,j}}$ (with $i \neq l$) along with 
a wedge $\phi_{\al}\wedge\phi_{\be}$ appearing in $d_1(\phi_{\si_{l,j}})$.
Note that $\al + \be = \si_{l,j}$.  Hence,
$$
\text{ht}\left(\bw_{i=2}^n\phi_{\si_i}\right) \leq
\text{ht}\left(\bw_{i = 3}^{n}\phi_{\si_{i,j}}\right) <
\text{ht}\left(\bw_{i = 2}^{n}\phi_{\si_{i,j}}\right) \leq
\text{ht}\left(\bw_{i=2}^n\phi_{\si_i}\right).
$$
Consequently,  this is also not possible.  Thus we must have $d_1(\phi_{\si_1})\wedge\bw_{i=2}^n\phi_{\si_i} = 0$.

If $d_1(\phi_{\si_1}) \neq 0$, then 
$d_1(\phi_{\si_1}) = \sum_{\al + \be = \si_1}c_{\al,\be}\phi_{\al}\wedge\phi_{\be}$.
As above, we have $\text{ht}(\al) < \text{ht}(\si_1) \leq \text{ht}(\si_i)$ for $2 \leq i \leq n$.
The analogous condition holds for $\be$ as well.  Hence, no $\al$ or $\be$ can
equal a $\si_i$, and we must have $d_1(\phi_{\si_1}) = 0$.
\end{proof}


\subsection{} Combining this with Proposition \ref{P:h1}, we get the following.

\begin{cor}\label{C:h3first} Assume that $p \geq 3$. For $p = 3$, assume further that
$\Phi$ is not of type $G_2$.
\begin{itemize}
\item[(a)] Let $x \in \opH^n(\ul,k)$ be a
representative cohomology class in reduced form having weight $\ga$ for
some $\ga \in X(T)$. Then one of the components of $x$ is of
the form $\bw_{i = 1}^{n}\phi_{\si_i}$ for distinct positive roots $\si_i \in \Phi^{+}$
with $\ga = \sum_{i=1}^{n}\si_i$ and at least one $\si_i$ being simple. 
\item[(b)] Suppose $\phi_{\al}\wedge\phi_{\be}\wedge\phi_{\de}$
represents a cohomology class in $\opH^3(\ul,k)$. By reordering if necessary, we
have the following conditions on $\al$, $\be$, and $\de$:
\begin{itemize}
\item[(i)] $\al$ is a simple root.
\item[(ii)] Either  $\be$ is a simple root or $\be = \al + \si$ for some $\si \in \Phi^{+}$,
and this is the unique decomposition of $\be$ as a sum of positive roots.
\item[(iii)] Either $\de$ is a simple root (in which case $\be$ is also) or $\de$ is of 
the form $\al + \si_1$, $\de = \be + \si_2$, or both for some $\si_i \in \Phi^{+}$, and those are the only possible 
decompositions of $\de$ as a sum of positive roots.
\end{itemize}
\end{itemize}
\end{cor}

\begin{proof} Part (a) follows immediately from Propositions \ref{P:h1} and \ref{P:d1}.
For part (b), by reordering as needed, we may assume that
$\text{ht}(\al) \leq \text{ht}(\be) \leq \text{ht}(\de)$.
From part (a), we have that that $\al$ must be simple. 

By assumption
$d_3(\phi_{\al}\wedge\phi_{\be}\wedge\phi_{\de}) = 0$. We have
\begin{align*}
d_3(\phi_{\al}\wedge\phi_{\be}\wedge\phi_{\de}) &= d_1(\phi_{\al})\wedge\phi_{\be}\wedge\phi_{\de} -
	\phi_{\al}\wedge d_1(\phi_{\be})\wedge\phi_{\de} + 
	\phi_{\al}\wedge\phi_{\be}\wedge d_1(\phi_{\de})\\
	& = - \phi_{\al}\wedge d_1(\phi_{\be})\wedge{\phi_{\de}}
		+ \phi_{\al}\wedge\phi_{\be}\wedge d_1(\phi_{\de}).
\end{align*}
In order for this to be zero, either both terms are independently zero or the terms
cancel each other out.  However, the first wedge product contains $\phi_{\de}$ and
the second wedge product can never contain $\phi_{\de}$ since the roots $\al, \be, \de$
are necessarily distinct. Thus the latter scenario
is impossible.  In other words, we must have both
$$
\phi_{\al}\wedge d_1(\phi_{\be})\wedge{\phi_{\de}} = 0 
\hskip.5in \text{ and } \hskip.5in
\phi_{\al}\wedge\phi_{\be}\wedge d_1(\phi_{\de}) = 0.
$$

Consider the first wedge.  If $d_1(\phi_{\be}) = 0$, then by Proposition 2.2, $\be$ is
simple.  Otherwise, we have
$$
d_1(\phi_{\be}) = \sum_{\si_1 + \si_2 = \be}c_{\si_1,\si_2}\phi_{\si_1}\wedge\phi_{\si_2}
$$
where the sum is over all distinct decompositions of $\be$ as a sum of positive roots.
Note further that under the hypotheses, all coefficients $c_{\si_1,\si_2}$ are non-zero
mod $p$. Therefore, the only way to have 
$\phi_{\al}\wedge d_1(\phi_{\be})\wedge{\phi_{\de}} = 0$ is 
if $d_1(\phi_{\be})$ involves $\phi_{\al}$ or $\phi_{\de}$. 
However, by our height assumption, the latter case is not possible.
That is, there is a unique decomposition of $\be$ into a sum of positive roots
and it has the form $\be = \al + \si$ for some $\si \in \Phi^{+}$. An analogous argument gives the constraints on $\de$.
\end{proof}


\subsection{} We make some further observations about weights that
arise in $\opH^3(\ul,k)$. Recall that our goal is to show that, for sufficiently large primes, all weights
of $\opH^3(\ul,k)$ have the form $-w\cdot 0$ for $w \in W$ with $\ell(w) = 3$.
We have also noted above that weights $-w\cdot 0$ with $\ell(w) = 3$ have the form $i\be_1 + j\be_2 + \be_3$
or $i\be_1 + j\be_2$ for some $\be_i \in \Pi$ and positive integers $i, j$.  
In this section, we show conversely that if such a weight arises as the weight
of a cohomology class, then it must indeed equal $-w\cdot 0$.

To show this, we first need some elementary computations for rank 3 root systems about the differential
$d_1 : \ul^* \to \Lambda^2(\ul^*)$.  Note that these computations also hold as appropriate
for rank 2 subsystems. We leave these computations as a straightforward exercise.
Note that these results are unique only up to a consistent sign change.

\begin{prop}\label{P:rank3} Assume that $\Phi$ is of type $A_3$, $B_3$, or $C_3$.  Let the simple
roots be $\al_1, \al_2, \al_3$ ordered in the natural way.  So $\al_3$ is the short
simple root in type $B_3$ and the long simple root in type $C_3$. Then we have
the following for $d_1: \Lambda^1(\ul^*) \to \Lambda^2(\ul^*)$.
\begin{itemize}
\item[(a)] All types:
\begin{itemize}
\item[$\bullet$] $d_1(\phi_{\al_1 + \al_2}) = \phi_{\al_1}\wedge\phi_{\al_2}$
\item[$\bullet$] $d_1(\phi_{\al_2 + \al_3}) = \phi_{\al_2}\wedge\phi_{\al_3}$
\item[$\bullet$] $d_1(\phi_{\al_1 + \al_2 + \al_3}) = 
	\phi_{\al_1}\wedge\phi_{\al_2 + \al_3} + \phi_{\al_1 + \al_2}\wedge\phi_{\al_3}$
\end{itemize}
\item[(b)] Type $B_3$:
\begin{itemize}
\item[$\bullet$] $d_1(\phi_{\al_1 + 2\al_2 + 2\al_3}) = 
	\phi_{\al_2}\wedge\phi_{\al_1 + \al_2 + 2\al_3} + 
	\phi_{\al_2 + 2\al_3}\wedge\phi_{\al_1 + \al_2} + 
	2\cdot\phi_{\al_2 + \al_3}\wedge\phi_{\al_1 + \al_2 + \al_3}$
\item[$\bullet$] $d_1(\phi_{\al_1 + \al_2 + 2\al_3}) =
	\phi_{\al_1}\wedge\phi_{\al_2 + 2\al_3} + 2\cdot\phi_{\al_3}\wedge\phi_{\al_1 + \al_2 + \al_3}$
\item[$\bullet$] $d_1(\phi_{\al_2 + 2\al_3}) = 
	2\cdot\phi_{\al_3}\wedge\phi_{\al_2 + \al_3}$
\end{itemize}
\item[(c)] Type $C_3$:
\begin{itemize}
\item[$\bullet$] $d_1(\phi_{2\al_1 + 2\al_2 + \al_3}) = 
	2\cdot\phi_{\al_1}\wedge\phi_{\al_1 + 2\al_2 + \al_3} + 
	2\cdot\phi_{\al_1 + \al_2}\wedge\phi_{\al_1 + \al_2 + \al_3}$
\item[$\bullet$] $d_1(\phi_{\al_1 + 2\al_2 + \al_3}) = 
	\phi_{\al_1}\wedge\phi_{2\al_2 + \al_3} + 
	\phi_{\al_2}\wedge\phi_{\al_1 + \al_2 + \al_3} +
	\phi_{\al_1 + \al_2}\wedge\phi_{\al_2 + \al_3}$
\item[$\bullet$] $d_1(\phi_{2\al_2 + \al_3}) = 
	2\cdot\phi_{\al_2}\wedge\phi_{\al_2 + \al_3}$
\end{itemize}
\end{itemize}
\end{prop}

Note that $\Phi$ necessarily has rank at least 3 in the following. In particular, type $G_2$ is not under consideration.

\begin{prop}\label{P:dotthree} Assume that $p \geq 3$.  If $p = 3$, assume further that
$\Phi$ is not of type $B_n$, $C_n$, or $F_4$. Suppose $\ga \in X(T)$ is a weight of 
$\opH^3(\ul,k)$ 
with $\ga = i\be_1 + j\be_2 + \be_3$ for distinct $\be_1,\be_2, \be_3 \in \Pi$ and 
$i \geq j > 0$.
Then $\ga = -(s_{\be_1}s_{\be_2}s_{\be_3})\cdot 0 = 
	s_{\be_1}s_{\be_2}(\be_3) + s_{\be_1}(\be_2) + \be_1$.
\end{prop}

\begin{proof} We proceed with a case-by-case analysis.  By assumption,
$\ga = \si_1 + \si_2 + \si_3$ where $\{\si_i\}$ consists of distinct positive 
roots. Furthermore, for each $i$, $\si_i = a\be_1 + b\be_2 + c\be_3$ for
some $a, b \geq 0$, and $c \in \{0,1\}$.  

\noindent{\bf Types $A_n$, $D_n$, and $E_n$:}
For these types, the coefficients $a, b$ must also lie 
in the set $\{0,1\}$.  The possible sums of three distinct positive roots giving
the desired form are as follows:
\begin{itemize}
\item[I.] $i = 3$, $j = 2$: $\ga = (\be_1 + \be_2 + \be_3) + (\be_1 + \be_2) + \be_1$
\item[II.] $i = 3$, $j = 1$: $\ga = (\be_1 + \be_2)$ + $(\be_1 + \be_3)$ + $\be_1$
\item[III.] $i = 2$, $j = 2$:
\begin{itemize}
\item[(a)] $\ga = (\be_1 + \be_2 + \be_3) + \be_2 + \be_1$
\item[(b)] $\ga = (\be_1 + \be_2) + (\be_1 + \be_3) + \be_2$
\item[(c)] $\ga = (\be_1 + \be_2) + (\be_2 + \be_3) + \be_1$
\end{itemize}
\item[IV.] $i = 2$, $j = 1$:
\begin{itemize}
\item[(a)] $\ga = (\be_1 + \be_2) + \be_3 + \be_1$
\item[(b)] $\ga = (\be_1 + \be_3) + \be_2 + \be_1$
\end{itemize}
\item[V.] $i = 1$, $j = 1$: $\ga = \be_3 + \be_2 + \be_1$
\end{itemize}

{\bf Case I.} Since there is only one way that the weight $\ga$ can arise,
the corresponding cohomology class must be represented by 
$\phi_{\be_1 + \be_2 + \be_3}\wedge\phi_{\be_1 + \be_2}\wedge\phi_{\be_1}$.
In particular, this element must be a cocycle.
From Lemma \ref{L:differential}, Proposition \ref{P:h1}, and Proposition \ref{P:rank3}, we have
\begin{align*}
0 &= d_3(\phi_{\be_1 + \be_2 + \be_3}\wedge\phi_{\be_1 + \be_2}\wedge\phi_{\be_1}) \\
	&= d_1(\phi_{\be_1 + \be_2 + \be_3})\wedge\phi_{\be_1 + \be_2}\wedge\phi_{\be_1}
	- \phi_{\be_1 + \be_2 + \be_3}\wedge d_1(\phi_{\be_1 + \be_2})\wedge\phi_{\be_1}\\
	& \qquad \quad + \phi_{\be_1 + \be_2 + \be_3}\wedge\phi_{\be_1 + \be_2}\wedge d_1(\phi_{\be_1})\\
	&= d_1(\phi_{\be_1 + \be_2 + \be_3})\wedge\phi_{\be_1 + \be_2}\wedge\phi_{\be_1}
		\pm \phi_{\be_1 + \be_2 + \be_3}\wedge\phi_{\be_1}\wedge\phi_{\be_2}\wedge\phi_{\be_1}\\
	&= d_1(\phi_{\be_1 + \be_2 + \be_3})\wedge\phi_{\be_1 + \be_2}\wedge\phi_{\be_1}.
\end{align*}
Since $\be_1 + \be_2$ is a positive root, the roots $\be_1$ and $\be_2$ are adjacent.
Further, since $\be_1 + \be_2 + \be_3$ is a positive root, $\be_3$ is adjacent to 
either $\be_1$ or $\be_2$.  That is, the roots correspond to nodes of the 
Dynkin diagram in one of two ways:
$$
\be_3\leftrightarrow\be_1\leftrightarrow\be_2 \hskip.5in \text{ or } \hskip.5in 
\be_1\leftrightarrow\be_2\leftrightarrow\be_3.
$$
In the first case, by Lemma~\ref{L:differential}, we have
\begin{align*}
d_1(\phi_{\be_1 + \be_2 + \be_3})\wedge\phi_{\be_1 + \be_2}\wedge\phi_{\be_1}
	&= \phi_{\be_3}\wedge\phi_{\be_1 + \be_2}\wedge\phi_{\be_1 + \be_2}\wedge\phi_{\be_1}
		+ \phi_{\be_3 + \be_1}\wedge\phi_{\be_2}\wedge\phi_{\be_1 + \be_2}\wedge\phi_{\be_1}\\
	&=  \phi_{\be_3 + \be_1}\wedge\phi_{\be_2}\wedge\phi_{\be_1 + \be_2}\wedge\phi_{\be_1} \neq 0,
\end{align*}
which is a contradiction.  Hence,  the latter situation must hold.  Then one can readily
check that $3\be_1 + 2\be_2 + \be_3 = -s_{\be_1}s_{\be_2}s_{\be_3}\cdot 0$.

{\bf Case II.} As in Case I, since the weight is unique, it must correspond to the
element $\phi_{\be_1 + \be_2}\wedge\phi_{\be_1 + \be_3}\wedge\phi_{\be_1}$.  Since
$\be_1 + \be_2$ and $\be_1 + \be_3$ are positive roots, $\be_1$ is adjacent to 
both $\be_2$ and $\be_3$. That is, we have $\be_2\leftrightarrow\be_1\leftrightarrow\be_3$.
With that condition, one can check that $3\be_1 + \be_2 + \be_3 = -s_{\be_1}s_{\be_2}s_{\be_3}\cdot 0$.

{\bf Case III.} Here, we potentially have three ways in which the weight $\ga$ 
could arise.  Notice that for (a), $\be_1$, $\be_2$, and $\be_3$ lie in a row in some
order. For (b), we must have $\be_2\leftrightarrow\be_1\leftrightarrow\be_3$, and
for (c) we have $\be_1\leftrightarrow\be_2\leftrightarrow\be_3$. So cases (b)
and (c) cannot occur simultaneously.  There are three scenarios to consider
up to a flip of the Dynkin diagram.

Suppose first that $\be_1\leftrightarrow\be_3\leftrightarrow\be_2$.  Then neither
case (b) or (c) holds and $\ga$ arises uniquely corresponding to 
$\phi_{\be_1 + \be_2 + \be_3}\wedge\phi_{\be_2}\wedge\phi_{\be_1}$.  Moreover,
one can check that $2\be_1 + 2\be_2 + \be_3 = -s_{\be_1}s_{\be_2}s_{\be_3}\cdot 0$.

Suppose next that we have $\be_2\leftrightarrow\be_1\leftrightarrow\be_3$. Then
$\ga$ can arise in two ways - (a) or (b).  Notice however that 
$\ga \neq -s_{\be_1}s_{\be_2}s_{\be_3}\cdot 0$ (indeed, 
$-s_{\be_1}s_{\be_2}s_{\be_3}\cdot 0$ is the weight arising in Case II). 
So we need to show that in fact there is no cohomology class having weight $\ga$.
If there was, $\ga$ would have to correspond to a linear
combination $a\phi_{\be_1 + \be_2 + \be_3}\wedge\phi_{\be_2}\wedge\phi_{\be_1}
+ b\phi_{\be_1 + \be_2}\wedge\phi_{\be_1 + \be_3}\wedge\phi_{\be_2}$. Then we have
(rewriting to match the ordering of the roots)
\begin{align*}
d_3(a\phi_{\be_2 + \be_1 + \be_3}&\wedge\phi_{\be_2}\wedge\phi_{\be_1}
+ b\phi_{\be_2 + \be_1}\wedge\phi_{\be_1 + \be_3}\wedge\phi_{\be_2}) \\
	&= a\phi_{\be_2}\wedge\phi_{\be_1 + \be_3}\wedge\phi_{\be_2}\wedge\phi_{\be_1} + 
		a\phi_{\be_2 + \be_1}\wedge\phi_{\be_3}\wedge\phi_{\be_2}\wedge\phi_{\be_1}\\
	&\quad + b\phi_{\be_2}\wedge\phi_{\be_1}\wedge\phi_{\be_1+\be_3}\wedge\phi_{\be_2} - 
		b\phi_{\be_2 + \be_1}\wedge\phi_{\be_1}\wedge\phi_{\be_3}\wedge\phi_{\be_2}\\
	&= a\phi_{\be_2 + \be_1}\wedge\phi_{\be_3}\wedge\phi_{\be_2}\wedge\phi_{\be_1} -
		b\phi_{\be_2 + \be_1}\wedge\phi_{\be_1}\wedge\phi_{\be_3}\wedge\phi_{\be_2}\\
	& = (a - b)\phi_{\be_2 + \be_1}\wedge\phi_{\be_3}\wedge\phi_{\be_2}\wedge\phi_{\be_1}.
\end{align*}
So this is a cocyle if and only if $a = b$.  That is the potential 
cohomology class would be represented by
$\phi_{\be_2 + \be_1 + \be_3}\wedge\phi_{\be_2}\wedge\phi_{\be_1}
+ \phi_{\be_2 + \be_1}\wedge\phi_{\be_1 + \be_3}\wedge\phi_{\be_2}$.
Notice however that
\begin{align*}
d_2(\phi_{\be_2 + \be_1 + \be_3}&\wedge\phi_{\be_2 + \be_1}) \\
	&= \phi_{\be_2}\wedge\phi_{\be_1 + \be_3}\wedge\phi_{\be_2 + \be_1}
		+ \phi_{\be_2 + \be_1}\wedge\phi_{\be_3}\wedge\phi_{\be_2 + \be_1}
			- \phi_{\be_2 + \be_1 + \be_3}\wedge\phi_{\be_2}\wedge\phi_{\be_1}\\
	& = \phi_{\be_2}\wedge\phi_{\be_1 + \be_3}\wedge\phi_{\be_2 + \be_1}
		- \phi_{\be_1 + \be_2 + \be_3}\wedge\phi_{\be_2}\wedge\phi_{\be_1}\\
	& = - \phi_{\be_1 + \be_2 + \be_3}\wedge\phi_{\be_2}\wedge\phi_{\be_1} 
		- \phi_{\be_2 + \be_1}\wedge\phi_{\be_1 + \be_3}\wedge\phi_{\be_2}.
\end{align*}
And so the above class is a coboundary. Hence, there is no cohomology class
of weight $\ga$ as claimed.  

Suppose finally that we have $\be_1\leftrightarrow\be_2\leftrightarrow\be_3$. 
Then $\ga$ can arise in two ways - (a) or (c).  As in the preceding case,
$\ga \neq -s_{\be_1}s_{\be_2}s_{\be_3}\cdot 0$ (indeed, 
$-s_{\be_1}s_{\be_2}s_{\be_3}\cdot 0$ is the weight arising in Case I). 
So we need to show that in fact there is no cohomology class having weight $\ga$.
If there was, $\ga$ would have to correspond to a linear
combination $a\phi_{\be_1 + \be_2 + \be_3}\wedge\phi_{\be_2}\wedge\phi_{\be_1}
+ b\phi_{\be_1 + \be_2}\wedge\phi_{\be_2 + \be_3}\wedge\phi_{\be_1}$. Then we have
(rewriting to match the ordering of the roots)
\begin{align*}
d_3(a\phi_{\be_1 + \be_2 + \be_3}&\wedge\phi_{\be_1}\wedge\phi_{\be_2}
+ b\phi_{\be_1 + \be_2}\wedge\phi_{\be_2 + \be_3}\wedge\phi_{\be_1}) \\
	&= a\phi_{\be_1}\wedge\phi_{\be_2 + \be_3}\wedge\phi_{\be_1}\wedge\phi_{\be_2} + 
		a\phi_{\be_1 + \be_2}\wedge\phi_{\be_3}\wedge\phi_{\be_1}\wedge\phi_{\be_2}\\
	&\quad + b\phi_{\be_1}\wedge\phi_{\be_2}\wedge\phi_{\be_2+\be_3}\wedge\phi_{\be_1} - 
		b\phi_{\be_1 + \be_2}\wedge\phi_{\be_2}\wedge\phi_{\be_3}\wedge\phi_{\be_1}\\
	&= a\phi_{\be_1 + \be_2}\wedge\phi_{\be_3}\wedge\phi_{\be_1}\wedge\phi_{\be_2} -
		b\phi_{\be_1 + \be_2}\wedge\phi_{\be_2}\wedge\phi_{\be_3}\wedge\phi_{\be_1}\\
	& = (a - b)\phi_{\be_1 + \be_2}\wedge\phi_{\be_1}\wedge\phi_{\be_2}\wedge\phi_{\be_3}.
\end{align*}

So this is a cocyle if and only if $a = b$.  That is the potential 
cohomology class would be represented by
$\phi_{\be_1 + \be_2 + \be_3}\wedge\phi_{\be_1}\wedge\phi_{\be_2}
+ \phi_{\be_1 + \be_2}\wedge\phi_{\be_2 + \be_3}\wedge\phi_{\be_1}$.
Notice however that
\begin{align*}
d_2(\phi_{\be_1 + \be_2 + \be_3}&\wedge\phi_{\be_1 + \be_2}) \\
	&= \phi_{\be_1}\wedge\phi_{\be_2 + \be_3}\wedge\phi_{\be_1 + \be_2}
		+ \phi_{\be_1 + \be_2}\wedge\phi_{\be_3}\wedge\phi_{\be_1 + \be_2}
			- \phi_{\be_1 + \be_2 + \be_3}\wedge\phi_{\be_1}\wedge\phi_{\be_2}\\
	& = \phi_{\be_1}\wedge\phi_{\be_2 + \be_3}\wedge\phi_{\be_1 + \be_2}
		- \phi_{\be_1 + \be_2 + \be_3}\wedge\phi_{\be_1}\wedge\phi_{\be_2}\\
	& = - \phi_{\be_1 + \be_2 + \be_3}\wedge\phi_{\be_1}\wedge\phi_{\be_2} 
		- \phi_{\be_1 + \be_2}\wedge\phi_{\be_2 + \be_3}\wedge\phi_{\be_1}.
\end{align*}
And so the above class is a coboundary. Hence, there is no cohomology class
of weight $\ga$ as claimed.

{\bf Case IV.} Notice that in case (a) $\be_1$ is adjacent to $\be_2$, and
in case (b) $\be_1$ is adjacent to $\be_3$.  Notice also that both
$\phi_{\be_1 + \be_2}\wedge\phi_{\be_1}\wedge\phi_{\be_3}$ and
$\phi_{\be_1 + \be_3}\wedge\phi_{\be_1}\wedge\phi_{\be_2}$ are cocycles.  

Suppose first that $\be_1$ is adjacent to both $\be_2$ and $\be_3$. Then we have
$\be_2\leftrightarrow\be_1\leftrightarrow\be_3$ and 
$\ga \neq -s_{\be_1}s_{\be_2}s_{\be_3}\cdot 0$. That would be Case II. 
So we want to show that there is no cohomology class having weight $\ga$.
Indeed, observe that
\begin{align*}
d_2(&\phi_{\be_2 + \be_1 + \be_3}\wedge\phi_{\be_1} + 
	\phi_{\be_2 + \be_1}\wedge\phi_{\be_1 + \be_3})\\
	&= \phi_{\be_2}\wedge\phi_{\be_1 + \be_3}\wedge\phi_{\be_1}
		+ \phi_{\be_2 + \be_1}\wedge\phi_{\be_3}\wedge\phi_{\be_1}
		+ \phi_{\be_2}\wedge\phi_{\be_1}\wedge\phi_{\be_1 + \be_3}
		- \phi_{\be_2 + \be_1}\wedge\phi_{\be_1}\wedge\phi_{\be_3}\\
	&= - 2\phi_{\be_2 + \be_1}\wedge\phi_{\be_1}\wedge\phi_{\be_3}.
\end{align*}
Hence, $\phi_{\be_2 + \be_1}\wedge\phi_{\be_1}\wedge\phi_{\be_3}$ is a coboundary.
Similarly, 
$$
d_2(\phi_{\be_2 + \be_1 + \be_3}\wedge\phi_{\be_1} 
	- \phi_{\be_2 + \be_1}\wedge\phi_{\be_1 + \be_3})
		= -2\phi_{\be_2 + \be_1}\wedge\phi_{\be_1}\wedge\phi_{\be_3}.
$$
And so $\phi_{\be_2 + \be_1}\wedge\phi_{\be_1}\wedge\phi_{\be_3}$ is also a coboundary.
Hence, $\ga$ cannot be the weight of a cohomology class.

Next assume that $\be_1$ is adjacent to $\be_2$ but not adjacent to $\be_3$.
Then we need only consider case (a) in which $\ga$ must correspond to 
$\phi_{\be_1 + \be_2}\wedge\phi_{\be_1}\wedge\phi_{\be_3}$.  There are still
two cases to consider depending upon whether $\be_3$ is adjacent to $\be_2$.
If it is, then we have $\be_1\leftrightarrow\be_2\leftrightarrow\be_3$ and
$\ga \neq -s_{\be_1}s_{\be_2}s_{\be_3}\cdot 0$.  This should be Case I. 
However, we have
\begin{align*}
d_2(\phi_{\be_1 + \be_2 + \be_3}\wedge\phi_{\be_1}) 
	&= \phi_{\be_1}\wedge\phi_{\be_2 + \be_3}\wedge\phi_{\be_1} 
		+ \phi_{\be_1 + \be_2}\wedge\phi_{\be_3}\wedge\phi_{\be_1}\\
	&= -\phi_{\be_1 + \be_2}\wedge\phi_{\be_1}\wedge\phi_{\be_3}.
\end{align*}
And so $\phi_{\be_1 + \be_2}\wedge\phi_{\be_1}\wedge\phi_{\be_3}$ is indeed
a coboundary. On the other hand, if $\be_3$ is not adjacent to $\be_2$ (nor $\be_1$),
then one can check that $2\be_1 + \be_2 + \be_3 = -s_{\be_1}s_{\be_2}s_{\be_3}\cdot 0$.

Finally, assume that $\be_1$ is adjacent to $\be_3$ but not to $\be_2$.  Then
we only need to consider case (b) in which $\ga$ must correspond to 
$\phi_{\be_1  + \be_3}\wedge\phi_{\be_1}\wedge\phi_{\be_2}$.  Similar to the 
preceding case, if $\be_2$ is adjacent to $\be_3$ (i.e., we have 
$\be_1\leftrightarrow\be_3\leftrightarrow\be_2$), one sees that 
$\phi_{\be_1  + \be_3}\wedge\phi_{\be_1}\wedge\phi_{\be_2}$ is a coboundary.
On the other hand, if $\be_2$ is not adjacent to $\be_3$ (nor $\be_1$), 
then $2\be_1 + \be_2 + \be_3 = -s_{\be_1}s_{\be_2}s_{\be_3}\cdot 0$.

{\bf Case V.} Here $\ga$ must correspond to 
$\phi_{\be_1}\wedge\phi_{\be_2}\wedge\phi_{\be_3}$ which is evidently a cocycle.
Observe that if any of the roots $\be_1$, $\be_2$, and $\be_3$ are adjacent, then
this class is a coboundary. For example, suppose $\be_1$ and $\be_2$ are 
adjacent. Then
$$
d_2(\phi_{\be_1 + \be_2}\wedge\phi_{\be_3}) = \phi_{\be_1}\wedge\phi_{\be_2}\wedge\phi_{\be_3}.
$$
The other cases are similar.  Hence, for this to represent a cohomology class, the three
roots must be completely disjoint.  Under that condition, one indeed has
$\be_1 + \be_2 + \be_3 = -s_{\be_1}s_{\be_2}s_{\be_3}\cdot 0$ as claimed.

\medskip\noindent
{\bf Type $B_n$:} If the none of the roots $\be_1$, $\be_2$, or $\be_3$
is equal to the short root $\al_n$, then the roots lie in the natural
root subsystem of type $A_{n-1}$ and the result follows from above.  Next,
if $\be_i \in \{\al_{n-2},\al_{n-1},\al_{n}\}$ for each $i$, then the
problem reduces to a Lie subalgebra of type $B_3$. For type $B_3$, the 
Lie algebra cohomology can be computed directly by hand (with a fair 
amount of work) or using computer programs and the software MAGMA \cite{BC,BCP}.
Such programs have been written by students at the University of Wisconsin-Stout
and by the VIGRE Algebra Group at the University of Georgia. Alternatively, since
$h = 6$, one can apply the general theory of \cite{FP2} \cite{PT} \cite{UGA}. One finds that
for $p \geq 5$, each cohomology class in $\opH^3(\ul,k)$ has weight
$w\cdot 0$ for some $w \in W$ with $l(w) = 3$.  Hence, the result follows.
Note that for $p = 3$, there do exist three ``extra'' cohomology classes,
two of which indeed have weight as given in the lemma. 

It remains to consider the case when one of the $\be_i$s is $\al_n$ and
another is $\al_j$ for some $j < n-2$.  Recall that the weight must arise
as a sum $\si_1 + \si_2 + \si_3$ of distinct positive roots.  We have the
following possibilities:
\begin{itemize}
\item[(i)] $\al_j + (\al_{n-1} + 2\al_n) + (\al_{n-1} + \al_n)$ for some $1 \leq j \leq n - 3$
\item[(ii)] $\al_j + (\al_{n-1} + 2\al_n) + \al_n$ for some $1 \leq j \leq n - 3$
\item[(iii)] $\al_j + (\al_{n-1} + 2\al_n) + \al_{n-1}$ for some $1 \leq j \leq n - 3$
\item[(iv)] $\al_j + (\al_{n-1} + \al_n) + \al_n$ for some $1 \leq j \leq n - 3$
\item[(v)] $\al_j + (\al_{n-1} + \al_n) + \al_{n-1}$ for some $1 \leq j \leq n - 3$
\item[(vi)] $\al_j + \al_{n-1} + \al_n$ for some $1 \leq j \leq n - 3$
\item[(vii)] $\al_j + \al_{n-2} + \al_n$ for some $1 \leq j \leq n - 4$
\item[(viii)] $\al_{n-3} + \al_{n-2} + \al_n$ 
\item[(ix)] $\si_1 + \si_2 + \al_n$ where $\si_1$ and $\si_2$ together involve precisely two
simple roots $\al_j$ with $1 \leq j \leq n - 3$.
\end{itemize}

For the first eight cases, the given weight arises uniquely as a sum of three distinct
positive roots.  Hence, the corresponding cohomology class would have to be 
$\phi_{\si_1}\wedge\phi_{\si_2}\wedge\phi_{\si_3}$.  In case (i), the corresponding element is not a cocyle.
Indeed, $\phi_{\al_j}\wedge\phi_{\al_{n-1} + 2\al_n}\wedge\phi_{\al_{n-1}+\al_n} 
\mapsto \phi_{\al_j}\wedge\phi_{\al_{n-1}+ 2\al_n}\wedge\phi_{\al_{n-1}}\wedge\phi_{\al_n}$.
In case (ii), the weight equals $-s_{\al_n}s_{\al_{n-1}}s_{\al_j}\cdot 0$ as needed.
In case (iii), the corresponding element is not a cocycle. Indeed,
$\phi_{\al_j}\wedge\phi_{\al_{n-1} + 2\al_n}\wedge\phi_{\al_{n-1}} 
\mapsto -2\phi_{\al_j}\wedge\phi_{\al_n}\wedge\phi_{\al_{n-1}+ \al_n}\wedge\phi_{\al_{n-1}}$.
In case (iv), the corresponding element is a cocycle but is also a coboundary.
Indeed,
$\phi_{\al_j}\wedge\phi_{\al_{n-1} + 2\al_n}\mapsto  2\phi_{\al_j}\phi_{\al_{n-1} + \al_n}\wedge\phi_{\al_n}$.
In case (v), the weight equals $-s_{\al_{n-1}}s_{\al_n}s_{\al_j}\cdot 0$ as needed.
In case (vi), the corresponding element is a cocycle but is also a coboundary.
Indeed,
$\phi_{\al_j}\wedge\phi_{\al_{n-1}+ \al_n} \mapsto - \phi_{\al_j}\wedge\phi_{\al_{n-1}}\wedge\phi_{\al_n}$.
In case (vii), the weight equals $-s_{\al_n}s_{\al_{n-2}}s_{\al_j}\cdot 0$ as needed.
In case (viii), the corresponding element is a cocyle but is also a cobouncary.
Indeed,
$\phi_{\al_{n-3} + \al_{n-2}}\wedge\phi_{\al_n} \mapsto \phi_{\al_{n-3}}\wedge\phi_{\al_{n-2}}\wedge\phi_{\al_n}$.

For case (ix), there are several subcases to consider:
\begin{itemize}
\item[(a)] $\si_1 = \al_j$ and $\si_2 = \al_{j + 1}$ with $1 \leq j \leq n-4$
\item[(b)] $\si_1 = \al_i$ and $\si_2 = \al_j$ with $1 \leq i < j - 1 \leq n - 4$
\item[(c)] $\si_1 = \al_j$ and $\si_2 = \al_j + \al_{j + 1}$ with $1 \leq j \leq n - 4$
\item[(d)] $\si_1 = \al_j$ and $\si_2 = \al_{j-1} + \al_j$ with $2 \leq j \leq n - 3$.
\end{itemize}
Again, these weights arise uniquely and all correpsond to cocycles.  However, in case (a),
the corresponding element is also a coboundary. Indeed,
$\phi_{\al_j + \al_{j+1}}\wedge\phi_{\al_n} \mapsto \phi_{\al_j}\wedge\phi_{\al_{j+1}}\wedge\phi_{\al_n}$.
In the latter three cases, the weight equals $-w\cdot 0$ as needed.  In order, the words are
$s_{\al_i}s_{\al_j}s_{\al_n}$, $s_{\al_j}s_{\al_{j + 1}}s_{\al_n}$, and $s_{\al_j}s_{\al_{j-1}}s_{\al_n}$.
Thus the claim holds for type $B_n$.

\medskip

\noindent
{\bf Type $C_n$:} For type $C_n$ one can argue similarly to type $B_n$.

\medskip\noindent
{\bf Type $F_4$:} Note that in this case $\opH^3(\ul,k)$ could be computed 
with the aid of a computer.  We present the general argument anyhow. 
If the simple roots $\{\be_1,\be_2,\be_3\}$ form a root subsystem of 
type $B_3$ or $C_3$, we are done by above.  That leaves the case of type $A_2 \times A_1$.
The possible root sums are
\begin{itemize}
\item[(i)] $\al_1 + \al_3 + \al_4$
\item[(ii)] $\al_1 + (\al_3 + \al_4) + \al_4$
\item[(iii)] $\al_1 + \al_3 + (\al_3 + \al_4)$
\item[(iv)] $\al_1 + \al_2 + \al_4$
\item[(v)] $\al_1 + (\al_1 + \al_2) + \al_4$
\item[(vi)] $(\al_1 + \al_2) + \al_2 + \al_4$
\end{itemize}
Note that all of these weights arise uniquely.  Cases (ii), (iii), (v), and (vi) all have
the desired form $-w\cdot 0$.  Specifically, in order, they are
$-s_1s_4s_3\cdot 0$, $-s_1s_3s_4\cdot 0$, $-s_1s_2s_4\cdot 0$, and $-s_2s_1s_4\cdot 0$.
In the remaining two cases, the corresponding elements are cocycles but also coboundaries.
Indeed, we have 
$\phi_{\al_1}\wedge\phi_{\al_3 + \al_4} \mapsto - \phi_{\al_1}\wedge\phi_{\al_3}\wedge\phi_{\al_4}$
and
$\phi_{\al_1 + \al_2}\wedge \phi_{\al_4} \mapsto \phi_{\al_1}\wedge\phi_{\al_2}\wedge\phi_{\al_4}$.
\end{proof}

\begin{prop}\label{P:dottwo} Assume that $p \geq 3$.  Assume that
$\Phi$ is not of type $G_2$. Suppose $\ga \in X(T)$ is a weight of 
$\opH^3(\ul,k)$ 
with $\ga = i\be_1 + j\be_2$ for $\be_1,\be_2 \in \Pi$ and 
$i \geq j > 0$.
Then $\ga = -(s_{\be_1}s_{\be_2}s_{\be_1})\cdot 0 = 
	s_{\be_1}s_{\be_2}(\be_1) + s_{\be_1}(\be_2) + \be_1$.
\end{prop}

\begin{proof} As in the proof of the preceding lemma, the weight $\ga$ must equal the sum 
of three distinct positive roots.  For types $A_n$, $D_n$, and $E_n$, the only way such a
weight can arise is if $\ga = \al_i + \al_{i+1} + (\al_i + \al_{i+1}) = 
-s_{\al_i}s_{\al_{i+1}}s_{\al_i}\cdot 0 = - s_{\al_{i+1}}s_{\al_{i}}s_{\al_{i+1}}\cdot 0$
as claimed.  In types $B_n$, $C_n$, and $F_4$, we have additional cases to consider when
working within a type $B_2$ (or equivalently $C_2$) root system.  For type $B_2$, there 
are four cases:
\begin{itemize}
\item[(i)] $\al_1 + \al_2 + (\al_1 + \al_2)$
\item[(ii)] $\al_1 + \al_2 + (\al_1 + 2\al_2)$
\item[(iii)] $\al_1 + (\al_+ + \al_2) + (\al_1 + 2\al_2)$
\item[(iv)] $\al_2 + (\al_1 + \al_2) + (\al_1 + 2\al_2)$.
\end{itemize}
Each of these weights arise uniquely.  In the first two cases, the corresponding elements
are coboundaries.  Indeed, we have
$\phi_{\al_1}\wedge\phi_{\al_1 + 2\al_2} \mapsto -2\phi_{\al_1}\wedge\phi_{\al_2}\wedge\phi_{\al_1 + \al_2}$
and
$\phi_{\al_1 + \al_2}\wedge\phi_{\al_1 + 2\al_2} \mapsto \phi_{\al_1}\wedge\phi_{\al_2}\wedge\phi_{\al_1+2\al_2}$.
The latter two weights are of the desired form: $-s_1s_2s_1\cdot 0$ and $-s_2s_1s_2\cdot 0$ respectively.
\end{proof} 

\begin{rem} In the case that $\Phi$ is of type $G_2$, all cohomology groups $\opH^i(\ul,k)$ may be computed by hand.  
For $p \geq 5$, $\opH^3(\ul,k)$ is two dimensional with weights corresponding to $-w\cdot 0$ for the two elements of 
length three.  However, for $p = 3$, there are six ``extra'' cohomology classes whose weights are not of the form
$-w\cdot 0$ for $\ell(w) = 3$.
\end{rem}


\section{Relationship between ordinary and restricted cohomology}


\subsection{} In this section, we investigate a spectral sequence relating $\opH^3(\ul,k)$ to $\opH^3(U_1,k)$.   In Section 4, this will be further related to $B_1$-cohomology, from which we will be able to make a precise determination of $\opH^3(\ul,k)$ for primes
that are not too small. See Theorem \ref{T:u-coho}.

By Corollary \ref{C:h3first}, we know that a weight $\ga$ of $\opH^3(\ul,k)$ can be expressed as the sum of three distinct positive roots, at least one of which is simple.  In studying these relationships between $\ul$-, $U_1$-, and $B_1$-cohomology, we make use of the fact that such root sums cannot take certain forms.   This will be the focus of the next two subsections.


\subsection{}{\bf Root sums I:} The following result shows that weights of $\opH^3(\ul,k)$ cannot lie in $pX(T)$ for $p \geq 5$.

\begin{prop}\label{P:d4rootsum} Assume $p \geq 5$. Let $\al, \si_1, \si_2 \in \Phi^+$ be distinct positive roots with $\al \in \Pi$. Then 
$\al + \si_1 + \si_2 \notin pX(T)$.
\end{prop}

\begin{proof} Set $\ga := \al + \si_1 + \si_2$ and suppose that $\ga = p\nu$ for some $\nu \in X(T)$.  Clearly $\nu \neq 0$.
Express $\ga$ as
a sum of fundamental dominant weights: $\ga = \sum c_i\omega_i$ for integers $c_i$.   By direct calculation, one finds the 
following constraints on the $c_i$ for any $i$.  

\vskip.1in
\begin{center} 
\begin{tabular}{|c|c|c|c|c|}
\hline
$\Phi$ & $A_n, D_n, E_n$ & $B_n$, $F_4$ & $C_n$ & $G_2$\\
\hline
Bounds on $c_i$ & $-3 \leq c_i \leq 4$ &  $-6 \leq c_i \leq 6$ & $-4 \leq c_i \leq 5$ & $-4 \leq c_i \leq 6$\\
\hline
\end{tabular}
\end{center}

\vskip.1in\noindent
For example, in types $A_n$, $D_n$, or $E_n$, when a positive root is expressed as a sum of fundamental dominant weights, the coefficients are $-1, 0, 1,$ or $2$, with 2 occurring only for a simple root.  Since we have a sum of three distinct roots, for any given $\omega_i$, the coefficient is at most $2 + 1 + 1 = 4$.  

Since $\nu \neq 0$, in a similar expression for $p\nu$, at least one $\omega_i$ has a coefficient that is a multiple of $p$.  From the above table, this is clearly not possible for $p \geq 7$.   Further, for types $A_n, D_n,$ and $E_n$ this is not possible for $p \geq 5$.   We are left to consider types $B_n$, $C_n$, $F_4$, and $G_2$ when $p = 5$.  In these cases, since $p\nu$ lies in the root lattice, $\nu$ necessarily also lies in the root lattice. In fact, $\nu$ must lie in the positive root lattice.  

Similar to above, express $\ga$ as a sum of simple roots: $\ga = \sum m_i\al_i$ with $m_i \geq 0$.  To have $\ga = 5\nu$, we must have each $m_i$ being divisible by 5.   For type $B_n$, the largest a coefficient can be is 5, and this occurs only if $\al = \al_n$ and $\si_1$, $\si_2$ both contain $2\al_n$. But then the coefficient of $\al_{n-1}$ is non-zero (it is at least two) and at most 4. So this is impossible.   In type $C_n$, one similarly could have $\al = \al_{n-1}$ and $\si_1, \si_2$ both containing $2\al_{n-1}$. But then the coefficient of $\al_n$ is non-zero (it is at least two) and not divisible by $5$.   One can argue similarly in types $F_4$ and $G_2$ or compute all possible values of $\ga$ to verify the claim.
\end{proof}

\begin{rem}  For $p = 2$ or $p = 3$, it is possible for such weights $\ga$ to lie in $pX(T)$.  For example, in type $B_2$, $\al_1 + (\al_1 + \al_2) + (\al_1 + 2\al_2) = 3(\al_1 + \al_2)$. 
\end{rem}


\subsection{}{\bf Root sums II:}  Next we consider whether weights of $\opH^3(\ul,k)$ can have the form $\si + p\nu$ for $\si$ in the 
positive root lattice and $\nu \neq 0$.  For sufficiently large primes, this is not possible.  For later computations, slightly more general statements are proven than just for the case that $\al, \si_1, \si_2$ are distinct positive roots.

\begin{lem}\label{L:rootsum} Let $\al \in \Pi$ and $\si_1, \si_2 \in \Phi^+ \cup \{0\}$.
Let $\si \in {\mathbb Z}\Phi^+$ with $\si = \sum n_i\al_i$ and $0 \leq n_i < p$.   
Make the following assumption on $p$ dependent on the root system:

\begin{center}
\begin{tabular}{|c|c|c|c|c|c|c|c|c|c|}
\hline
$\Phi$ & $A_n$ & $B_n$ & $C_n$ & $D_n$ & $E_6$ & $E_7$ & $E_8$ & $F_4$ & $G_2$\\
\hline
$p \geq $ & 5 & 7 & 7 & 7 & 11 & 11 & 17 & 11 & 11\\
\hline
\end{tabular}
\end{center}

\vskip.2in\noindent
Suppose $\al + \si_1 + \si_2 = \si + p\nu$ for some $\nu \in X(T)$. 
Suppose further that in type $A_n$ either $\nu$ lies in the root lattice or $p$ does not divide $n + 1$.  
Then $\nu = 0$.  
\end{lem}

\begin{proof}  First observe that $p\nu$ necessarily lies in the root lattice.  By the assumptions on $p$, this implies that $\nu$ lies in the root lattice.  

Set $\ga := \al + \si_1 + \si_2$ and express $\ga = \sum m_i\al_i$ as a sum of simple roots (with $m_i \geq 0$).   The equation $\ga = \si + p\nu$ becomes $\sum m_i\al_i = \sum n_i\al_i + p\nu$.  Since $0 \leq n_i < p$ for each $i$, if $p > m_i$ for all $i$, then we necessarily have $\nu = 0$.  We recall in the following table the maximum value of a coefficient when a root is expressed as a sum of simple roots.

\vskip.1in\begin{center}
\begin{tabular}{|c|c|c|c|c|c|c|c|c|c|}
\hline
$\Phi$ & $A_n$ & $B_n$ & $C_n$ & $D_n$ & $E_6$ & $E_7$ & $E_8$ & $F_4$ & $G_2$\\
\hline
max coefficient & 1 & 2 & 2 & 2 & 3 & 4 & 6 & 4 & 3\\
\hline
\end{tabular}
\end{center}

\vskip.1in\noindent
Keeping in mind that $\al$ is simple, the assumptions on $p$ indeed give $p > m_i$ for all $i$.
\end{proof}

For our purposes, the weight $\si$ appearing in Lemma \ref{L:rootsum} will have more specific form.  The next 
two results present cases where one can lower the condition on the prime. 

\begin{prop}\label{P:d2rootsum} Let $\al \in \Pi$, $\si_1, \si_2 \in \Phi^+ \cup \{0\}$, and assume that all 
non-zero roots are distinct.  Suppose that 
$$
\al + \si_1 + \si_2 = \be + \si_3 + p\nu
$$
for distinct roots $\be, \si_3 \in \Phi^+$ with $\be \in \Pi$ and $\nu \in {\mathbb Z}\Phi$. 
Make the following assumption on $p$ dependent on the root system:

\begin{center}
\begin{tabular}{|c|c|c|c|c|c|c|c|c|c|}
\hline
$\Phi$ & $A_n$ & $B_n$ & $C_n$ & $D_n$ & $E_n$ & $F_4$ & $G_2$\\
\hline
$p \geq $ & 5 & 7 & 5 & 5 & 5 & 7 & 7\\
\hline
\end{tabular}
\end{center}

\vskip.1in\noindent
Then $\nu = 0$.  
\end{prop}

\begin{proof}  By Lemma \ref{L:rootsum}, we only need to consider the following cases: types $C_n$, $D_n$ and $E_n$ for $p = 5$, 
types $E_n$, $F_4$, and $G_2$ for $p = 7$, and type $E_8$ for $p = 11, 13$.  We proceed in a case-by-case basis, and, as in 
the proof of Lemma \ref{L:rootsum}, express the left and right hand sides of the given equation as sums of simple roots.

Consider type $C_n$ with $p = 5$. To have a non-trivial solution, we would need to have $\al = \al_{n-1}$, both $\si_1$ and $\si_2$ containing $2\al_{n-1}$, $\nu = \al_{n-1}$, and neither $\be$ nor $\si_3$ containing $\al_{n-1}$.  Since $\si_3$ does not contain an $\al_{n-1}$, it cannot contain {\em both} an $\al_{n-2}$ and an $\al_n$, nor does $\be$ (since $\be$ is simple).  So the sum of the coefficients of $\al_{n-2}$ and $\al_n$ in $\be + \si_3$ is at most two. On the other hand, both $\si_1$ and $\si_2$ must contain an $\al_n$ and, since they are distinct, at least one contains an $\al_{n-2}$.  Therefore, the sum of the coefficients of $\al_{n-2}$ and $\al_n$ in $\si_1 + \si_2$ is at least three.  So no solution is possible.

For type $D_n$ with $p = 5$, similar to the type $C_n$ case, to have a non-trivial solution, 
we must have $\al = \al_{n-2}$, $\si_1, \si_2$ each containing $2\al_{n-2}$, and neither $\be$ nor $\si_3$ 
containing an $\al_{n-2}$. Then $\si_1 + \si_2$ contains $2\al_{n-1} + 2\al_n$, while $\be$ and $\si_3$ can each contain 
at most one of $\al_{n-1}$ or $\al_n$.  Therefore, no solution is possible.

For the exceptional cases, the claim has been verified by using Magma \cite{BC, BCP} to compute all possible sums over
$\al, \si_1, \si_2, \be, \si_3$.   One can also argue in a manner similar to above (with many more cases). 
\end{proof}

\begin{rem}\label{R:d2rootsum} When $p = 2, 3$, there exists cases with $\nu \neq 0$ in all types.  For $p = 5$, 
arguing as above, one finds that the {\it only} non-trivial solutions are as follows:
\begin{itemize}
\item[(i)] Type $B_n$: $n \geq 3$, and $1 \leq i \leq n - 2$, 
\begin{itemize}
\item[$\bullet$]
$\al_n + (\al_{n-1} + 2\al_n) + (\al_i + \cdots + \al_{n-1} + 2\al_n) = \al_{n-1} + (\al_i + \cdots + \al_{n-1}) + 5\al_n$ 
\item[$\bullet$] For $i < n-2$, the weight $\ga := \al_i + \cdots + \al_{n-2} + 2\al_{n-1} + 5\al_n$ does not have the form 
$-w\cdot 0$ for $\ell(w) = 3$.  Moreover, the only way that $\ga$ can arise as the sum of three distinct positive roots is as 
$\ga = \al_n + (\al_{n-1} + 2\al_n) + (\al_i + \cdots + \al_{n-1} + 2\al_n)$. As such, the corresponding cohomology class would 
have to be represented by $\phi_{\al_n}\wedge\phi_{\al_{n-1} + 2\al_n}\wedge\phi_{\al_i + \cdots + \al_{n-1} + 2\al_n} \in \Lambda^3(\ul^*)$.  However, one can directly check that the differential is non-zero on this element, and so it does not 
represent a cohomology class in $\opH^3(\ul,k)$. 
\end{itemize}
\item[(ii)] Type $F_4$:  
\begin{itemize}
\item[$\bullet$] $\al_3 + (\al_2 + 2\al_3) + (\al_1 + \al_2 + 2\al_3) = \al_2 + (\al_1 + \al_2) + 5\al_3$
\item[$\bullet$] $\al_3 + (\al_2 + 2\al_3) + (\al_1 + 2\al_2 + 4\al_3 + 2\al_2) = \al_2 + (\al_1 + 2\al_2 + 2\al_3 + 2\al_4) + 5\al_3$
\end{itemize}
\item[(iii)] Type $G_2$: 
\begin{itemize}
\item[$\bullet$] $\al_1 + (2\al_1 + \al_2) + (3\al_1 + \al_2) = \al_2 + (\al_1 + \al_2) + 5\al_1$.
\end{itemize}
\end{itemize}
\end{rem}

\begin{prop}\label{P:rootsum-triple} Let $\al, \si_1, \si_2 \in \Phi^+$ be distinct positive roots with $\al \in \Pi$.
Suppose that 
$$
\al + \si_1 + \si_2 = i_1\be_1 + i_2\be_2 + i_3\be_3 + p\nu 
$$
for distinct $\be_j \in \Pi$ with $0 \leq i_1, i_2 < p$, $0 \leq i_3 \leq 1$ and $\nu \in X(T)$.
Make the following assumption on $p$ dependent on the root system:

\vskip.1in
\begin{center}
\begin{tabular}{|c|c|c|c|c|c|c|c|c|c|c|}
\hline
$\Phi$ & $A_n, n \neq 4, 6$ & $A_4$ & $A_6$ & $B_2 = C_2$ & $B_n, n \geq 3$ & $C_n, n \geq 3$ & $D_n$ & $E_n$ & $F_4$ & $G_2$\\
\hline
$p \geq $ & 5 & 7 & 5 ($\neq 7$) & 5 & 7 & 7 & 5 & 5 & 11 & 11\\
\hline
\end{tabular}
\end{center}

\vskip.1in\noindent
Then $\nu = 0$.  
\end{prop}

\begin{proof} In type $A_n$ when $p$ does not divide $n + 1$, type $B_n$ ($n \geq 3$), type $C_n$ ($n \geq 3$), type $F_4$, 
and type $G_2$, we are done by Lemma \ref{L:rootsum}. Set $\ga := \al + \si_1 + \si_2$.

For type $B_2$ (or equivalently $C_2$) when $p = 5$, when 
expressed as a sum of simple roots, the maximal coefficient of $\ga$ is 4. Hence, no solution is possible.  
In type $D_n$, we only need to consider the case $p = 5$.  We may assume 
that $n \geq 5$ since in $D_4$ only one root has a coefficient of 2 when expressed as a sum of simple roots.  
As in the proof of Proposition~\ref{P:d2rootsum}, to have a non-trivial solution, we would need 
$\al = \al_{n-2}$, $\si_1, \si_2$ each containing $2\al_{n-2}$, and $\nu = \al_{n-2}$.
Then $\ga - 5\al_{n-2}$ would contain at least four simple roots, and we are allowed at most 3. 
For type $E_n$, the claim has again been verified by direct computation using MAGMA \cite{BC,BCP}.

Lastly, consider the case that $\Phi$ is of type $A_n$ with $p$ dividing $n + 1$.  We argue as in \cite[Section 3.4]{BNP2}.  
In this situation, $X(T)/{\mathbb Z}\Phi = \{t\omega_1 + {\mathbb Z}\Phi : t = 0, 1, \dots, n\}$ and 
$$
t\omega_1 = \frac{t}{n+1}(n\al_1 + (n-1)\al_2 + \cdots + \al_n).
$$
If $\nu$ lies in the root lattice, then we are done, so we may assume that $\nu = t\omega_1 + \nu'$ for some $t \neq 0$ and
$\nu' \in {\mathbb Z}\Phi$. Our equation becomes
$$
\ga = i_1\be_1 + i_2\be_2 + i_3\be_3 + pt\omega_1 + p\nu'.
$$
We may further assume that $pt/(n+1)$ is not congruent to zero mod $p$, or we would be done as previously.   
Express both sides as sums of simple roots.  
On the left, the coefficients appearing in $\ga$ can be $0, 1, 2,$ or $3$. In other words, there are at most three distinct non-zero
coefficients.   On the right, notice that, mod $p$, the numbers $1, 2, \dots, p-1$ all appear as coefficients in $pt\omega_1$.  
At most three of those can be ``cancelled'' by the $\be_j$ terms.  Hence, there are at least $p - 4$ distinct non-zero coefficients mod $p$.  
So we have a contradiction if $p - 4 > 3$ or $p > 7$.  For $p = 7$, if $n + 1 \geq 14$, notice that the numbers $1, 2, \dots, p-1$ appear at least {\it twice} as coefficients (from distinct simple roots) in $pt\omega_1$, which again leads to a contradiction. 
One has a similar situation for $p = 5$ and $n + 1 \geq 10$.   That leaves the two ``bad'' cases of type $A_6$ when $p = 7$ and $A_4$ when $p = 5$.  

Note that there are non-trivial solutions in these ``bad'' type $A$ cases. For example, in type $A_4$ (with $p = 5$) some
solutions (not all) appear in Remark~\ref{R:rootsum-simple} below.
\end{proof}

\begin{rem}\label{R:rootsum-triple} For the following root systems and primes, the following are the {\it only} non-trivial solutions:
\begin{itemize}
\item[(i)] Type $C_n$, $n \geq 3$, $p = 5$
\begin{itemize}
\item[$\bullet$] 
$\al_{n-1} + (c_1\al_{n-2} + 2\al_{n-1} + \al_n) + (c_2\al_{n-3} + c_3\al_{n-2} + 2\al_{n-1} + \al_n) = c_2\al_{n-3} + (c_1 + c_3)\al_{n-2} + 2\al_n + 5\al_{n-1}$ for appropriate $c_1 \in \{0,1,2\}$, $c_2 \in \{0,1\}$, $c_3 \in \{1, 2\}$ (9 cases for $n \geq 4$, 3 cases for $n = 3$)
\item[$\bullet$] 
$\al_{n-2} + (2\al_{n-2} + 2\al_{n-1} + \al_n) + (\al_{n-3} + 2\al_{n-2} + 2\al_{n-1} + \al_n) = 
\al_{n-3} + 4\al_{n-1} + 2\al_n + 5\al_{n-2}$ ($n \geq 4$)
\item[$\bullet$]
Note that all weights occur within the type $C_4$ root subsystem.
\end{itemize}
\item Type $F_4$, $p = 7$
\begin{itemize}
\item[$\bullet$]
$\al_3 + (\al_2 + 2\al_3 + c\al_4) + (\al_1 + 2\al_2 + 4\al_3 + 2\al_4) = \al_1 + 3\al_2 + (c + 2)\al_4 + 7\al_3$ for $c \in \{0,1,2\}$
\item[$\bullet$]
$\al_3 + (\al_2 + 2\al_3 + c\al_4) + (\al_1 + 3\al_2 + 4\al_3 + 2\al_4) = \al_1 + 4\al_2 + (c + 2)\al_4 + 7\al_3$ for $c \in \{0,1,2\}$
\item[$\bullet$]
$\al_2 + (\al_1 + 3\al_2 + 4\al_3 + 2\al_4) + (2\al_1 + 3\al_2 + 4\al_3 + 2\al_4) = 3\al_1 + \al_3 + 4\al_4 + 7(\al_2 + \al_3)$
\end{itemize}
\item Type $G_2$, $p = 7$
\begin{itemize}
\item[$\bullet$] $\al_1 + (3\al_1 + \al_2) + (3\al_1 + 2\al_2) = 3\al_2 + 7\al_1$
\end{itemize}
\item Denote the above expressions by $\ga = \si + p\nu$.  Note that in each case neither $\ga$ nor $\si$ has the form
$-w\cdot 0$ for $\ell(w) = 3$.  
\end{itemize}
\end{rem}

\begin{cor}\label{C:rootsum-simple} Assume $p \geq 5$.  If $\Phi$ is of type $A_4$, assume $p \geq 7$.
Let $\al, \si_1, \si_2 \in \Phi^+$ be distinct positive roots with $\al \in \Pi$. Then
$$
\al + \si_1 + \si_2 \neq \be + p\nu
$$
for any $\be \in \Pi$ and $\nu \in X(T)$.
\end{cor}

\begin{proof} First observe that we could not have $\nu = 0$.  
In types $A_n$ (except $A_6$ with $p = 7$), $D_n$, and $E_n$, the claim follows immediately from 
Proposition~\ref{P:rootsum-triple}.   By direct calculation using MAGMA \cite{BC, BCP}, the claim can be verified for type
$A_6$ with $p = 7$ and types $F_4$ and $G_2$ when $p = 5, 7$ (with larger primes following from the proposition).  It remains
to consider types $B_n$ and $C_n$ ($n \geq 3$) for $p = 5$.  In those cases, we argue as in the proof of 
Proposition~\ref{P:d2rootsum} supposing that there was a non-zero $\nu$ for which the equation held.  As seen there,
in type $B_n$ (respectively, $C_n$), for that to be possible, the left side would contain between 
$2\al_{n-1}$ and $4\al_{n-1}$ (respectively, $2\al_n$ and $4\al_n$) which would not match with the single simple root $\be$
on the right.
\end{proof}

\begin{rem}\label{R:rootsum-simple} For type $A_4$ with $p = 5$, there are precisely two solutions
\begin{itemize}
\item $-s_3s_2s_1\cdot 0 = \al_3 + (\al_2 + \al_3) + (\al_1 + \al_2 + \al_3) = \al_4 + 5(\omega_3 - \omega_4)$
\item $-s_2s_3s_4\cdot 0 = \al_2 + (\al_2 + \al_3) + (\al_2 + \al_3 + \al_4) = \al_1 + 5(-\omega_1 + \omega_2)$.
\end{itemize}

\end{rem}


\subsection{\bf Relating $\ul$ and $U_1$:} We can relate  
$\opH^3(\ul,k)$ to $\opH^3(U_1,k)$ via the first quadrant 
spectral sequence introduced by Friedlander and Parshall \cite{FP1}
for $p \geq 3$ (and later generalized by Andersen and Jantzen \cite{AJ} and
Friedlander and Parshall \cite{FP2}; cf.~also \cite[I.9.20]{Jan1}):

\begin{equation}
E_2^{2i,j} = S^i(\ul^*)^{(1)}\otimes\opH^j(\ul,k) \Rightarrow
\opH^{2i + j}(U_1,k).
\end{equation}

The spectral sequence admits an action of the maximal torus $T$.  Under this
action, all the differentials are $T$-homomorphisms.
Since the odd columns are all zero, the only terms
that can contribute to $\opH^3(U_1,k)$ are 
$E_2^{2,1} = (\ul^*)^{(1)}\otimes\opH^1(\ul,k)$ and $E^{0,3}_2 = \opH^3(\ul,k)$. 

The first question to consider is whether or not these terms (or what portion
of them) survive in the spectral sequence.  For the first term, we must
consider the differential 
$$
d_2 : (\ul^*)^{(1)}\otimes\opH^1(\ul,k) = E_{2}^{2,1} \to E_{2}^{4,0} = S^2(\ul^*)^{(1)}.
$$
We need to identify the kernel of this map.  From Proposition \ref{P:h1}, as long as $p \geq 3$
or $\Phi$ is not of type $G_2$ when $p = 3$, then elements in $(\ul^*)^{(1)}\otimes\opH^1(\ul,k)$
have weight $p\si + \al$ for some $\si \in \Phi^{+}$ and some simple root $\al \in \Pi$.
On the other hand, any weight of the righthand side is of the form $p\la$ for
$\la = \si_1 + \si_2$ with $\si_i \in \Phi^{+}$.  This implies that $\al = p\tau$
for some $\tau \in X(T)$ which is impossible for $p \geq 3$.  Therefore, all elements
must map to zero.  When $p = 3$ and $\Phi$ is of type $G_2$, the ``extra''
weight of $\opH^1(\ul,k)$ is $3\al_1 + \al_2$ which would similarly lead to the condition 
that $\al_2 = 3\tau$ which is impossible.  So for $p \geq 3$, 
$\operatorname{ker}(d_2) = (\ul^*)^{(1)}\otimes\opH^1(\ul,k)$.

Next, we consider the image of 
$$
d_2 : \opH^2(\ul,k) = E_{2}^{0,2} \to E_{2}^{2,1} = (\ul^*)^{(1)}\otimes\opH^1(\ul,k).
$$
This differential was studied in \cite{BNP2}. For $p > 3$, it was shown \cite[Prop. 4.1]{BNP2}
that the map is zero.  For $p = 3$, it was found that the map could be non-zero for
elements having certain weight in types $B_n$ ($\la = \al_{n-1} + 3\al_n$), 
$C_n$ ($\la = 3\al_{n-1} + \al_n$), $F_4$ ($\la = \al_2 + 3\al_3$), and 
$G_2$ ($\la = 3\al_1 + \al_2$).  Moreover, these weights arise uniquely on each side, and 
using \cite[Prop. 5.2]{BNP2} along with the fact that
$\opH^2(B_1,\la) = \opH^2(U_1,\la)^{T_1} = (\opH^2(U_1,k)\otimes\la)^{T_1}$,
one finds that the map is indeed non-zero on representative elements corresponding to these weights.

As noted in \cite{BNP2}, the spectral sequence can be
refined.  Since weight spaces are preserved by the differentials,
and all modules for $T$ are completely reducible, for each
$\lambda\in X(T)$, one obtains a spectral sequence:

\begin{equation}
E_2^{2i,j} = [S^i(\ul^*)^{(1)}\otimes\opH^j(\ul,k)]_{\la}
\Rightarrow
\opH^{2i + j}(U_1,k)_{\la}.
\end{equation}

From our analysis above, we obtain the following result.

\begin{prop} let $p \geq 3$ and $\la \in X(T)$.  As a $T$-module,
\begin{equation}
\opH^3(U_1,k)_{\la} \supseteq ((\ul^*)^{(1)}\otimes\opH^1(\ul,k))_{\la}
\end{equation}
except for the following weights
\begin{itemize}
\item[(i)] $p = 3$, $\Phi$ is of type $B_n$: $\la = \al_{n-1} + 3\al_n$,
\item[(ii)] $p = 3$, $\Phi$ is of type $C_n$: $\la = 3\al_{n-1} + \al_n$,
\item[(iii)] $p = 3$, $\Phi$ is of type $F_4$: $\la = \al_2 + 3\al_3$,
\item[(iv)] $p = 3$, $\Phi$ is of type $G_2$: $\la = 3\al_1 + \al_2$.
\end{itemize}
In the exceptional cases, $\opH^3(U_1,k)_{\la} \subseteq \opH^3(\ul,k)_{\la}$
as $T$-modules.
\end{prop}


\subsection{} We next turn to the investigation of the term $E_2^{0,3} = \opH^3(\ul,k)$.
In order to determine its contribution to $\opH^3(U_1,k)$, there are
two differentials to consider:
$$
d_2: \opH^3(\ul,k) = E_2^{0,3} \to E_2^{2,2} = (\ul^*)^{(1)}\otimes \opH^2(\ul,k)
$$
and then
$$
d_4: \opH^3(\ul,k) = E_2^{0,3} \to E_{2}^{4,0} = S^2(\ul^*)^{(1)}.
$$
From Corollary \ref{C:h3first}, we know that weights of $\opH^3(\ul,k)$ have the form $\al + \si_1 + \si_2$ for distinct positive roots $\al, \si_1, \si_2$ with $\al$ being simple. 

Consider the differential $d_2$.  On the right hand side, by Theorem \ref{T:h2}, the weights have the form $p\nu + \be + \si_3$ for positive roots $\nu, \be, \si_3$ with $\be$ simple and $\be \neq \si_3$.  To have a non-trivial differential, we must have $\al + \si_1 + \si_2 = p\nu + \be + \si_3$ (with $\nu \neq 0$).  Applying Proposition~\ref{P:d2rootsum}, this is impossible for the primes given therein. 

For those classes that survive to the $E_4$-page, consider now $d_4$.  Here, the image has to have weight of the form $p(\si_3 + \si_4)$ for positive roots $\si_3, \si_4$. In other words, $\al + \si_1 + \si_2 = p(\si_3 + \si_4)$. By Proposition~\ref{P:d4rootsum}, this is not possible for $p \geq 5$. We conclude the following.

\begin{thm}\label{T:utoU1} Let 
\begin{itemize} 
\item[(a)] $p\geq 5$ when $\Phi$ is of types $A_{n}$, $C_{n}$, $D_{n}$, $E_n$,
\item[(b)] $p\geq 7$ when $\Phi$ is of type $B_n$ ($n \geq 3$), $F_4$, or $G_2$. 
\end{itemize} 
Then, as $T$-modules,
$$
\opH^3(U_1,k) \cong \opH^3(\ul,k)\oplus((\ul^*)^{(1)}\otimes \opH^1(\ul,k)).
$$
\end{thm} 

\begin{rem} In the cases for $p=5$ when $\Phi$ is of type $B_n$ ($n \geq 3$), $F_4$, or $G_2$, one still has 
$\opH^3(U_1,k) \supset (\ul^*)^{(1)}\otimes \opH^1(\ul,k)$, however, $\opH^3(U_1,k)_{\la}$ may not contain $\opH^3(\ul,k)_{\la}$ for the following weights (cf. Remark~\ref{R:d2rootsum}):
\begin{itemize}
\item[(i)] Type $B_n$ ($n \geq 3$): $\la = \al_{n-2} + 2\al_{n-1} + 5\al_n = -s_ns_{n-1}s_{n-2}\cdot 0$, 
\item[(ii)] Type $F_4$: $\la = \al_1 + 2\al_2 + 5\al_3 = -s_3s_2s_1\cdot 0$ or $\la = \al_1 + 3\al_2 + 7\al_3 + 2\al_4$,
\item[(iii)] Type $G_2$: $\la = 6\al_1 + 2\al_2$.  
\end{itemize}
\end{rem}


\section{Utilizing the $B$-cohomology} 


\subsection{} From this point on, unless otherwise specified, we will make the following assumptions on 
the characteristic of the field $k$. 

\begin{assump} \label{A:char} Let $p$ be the characteristic of the field $k$. We impose the condition that 
\begin{itemize} 
\item[(a)] $p\geq 5$ when $\Phi$ is of types $A_{n}$, $C_{n}$, $D_{n}$, or $E_n$,
\item[(b)] $p \neq 5$ when $\Phi$ is of type $A_4$, $p \neq 7$ when $\Phi$ is of type $A_6$,
\item[(c)] $p\geq 7$ when $\Phi$ is of type $B_n$ ($n \geq 3$), $F_4$, or $G_2$. 
\end{itemize} 
\end{assump} 


\subsection{\bf Relating $\ul$ and $B_1$:} The following result relates the $B_{1}$-cohomology with the 
$\ul$-cohomology. 

\begin{prop}\label{P:B1-cohostructure}  Let $p$ satisfy Assumption~\ref{A:char} and $\la \in X(T)$. 
\begin{itemize} 
\item[(a)] As a $T$-module, 
$$
\opH^3(B_1, \la) \cong  \bigoplus_{\nu \in X(T)}   p \nu^{\dim \opH^3(\ul,k)_{-\la +p\nu}+\dim [((\ul^*)^{(1)}\otimes \opH^1(\ul,k)]_{-\la+p\nu}}
$$
\item[(b)] $\opH^{3}(B_{1},k)=0$. 
\end{itemize} 
\end{prop} 

\begin{proof}  (a) The same argument as in \cite[Prop. 4.2]{BNP2} gives that
$$
\opH^3(B_1,\la) \cong \bigoplus_{\nu \in X(T)}p\nu^{\dim\opH^3(U_1,k)_{-\la + p\nu}}.
$$
The claim now follows from Theorem~\ref{T:utoU1}. 

(b) We have from Theorem~\ref{T:utoU1} that $\opH^{3}(B_{1},k)\cong \opH^{3}(U_{1},k)^{T_{1}}\cong \opH^{3}(\ul,k)^{T_{1}}$. 
By Corollary~\ref{C:h3first} and Proposition~\ref{P:d4rootsum}, $\opH^3(\ul,k)^{T_1} = 0$. 
\end{proof}

\begin{rem} The proposition holds also for types $A_4$ and $A_6$ for all $p \geq 5$.
\end{rem}


\subsection{\bf $\mathfrak u$-cohomology:} In \cite{BNP2}, the authors made use of known
information on $\opH^j(B,\la)$ $j=0,1,2$ in order to compute $\opH^2(\ul,k)$. This strategy will be used in the following theorem. Information about 
$\opH^{j}(B,k)$ for $j=0,1,2,3$ will be used to determine $\opH^{3}({\mathfrak u},k)$. 
The following theorem extends the work of \cite{PT, FP2, UGA} which required $p \geq h - 1$. 

\begin{thm} \label{T:u-coho} Let $p$ be an odd prime.  In addition, assume that $p \geq 5$ when $\Phi$ is of type $A_n$ ($n \geq 4$), $B_3$, $C_n$ ($n \geq 3$), $D_n$, $E_n$, and $G_2$, and that $p \geq 7$ when $\Phi$ is of type $B_n$ ($n \geq 4$) and $F_4$. 
\begin{itemize} 
\item[(a)] If $\ga \in X(T)$ is a weight of $\opH^3(\ul,k)$, then $\ga = -w\cdot 0$ for
some $w \in W$ with $l(w) = 3$. 
\item[(b)] As a $T$-module,
$$
\opH^3(\ul,k) \cong \bigoplus_{w\in W,\ l(w)=3} -w\cdot 0
$$
\end{itemize}
\end{thm}

\begin{proof} By direct computation of $\opH^3(\ul,k)$, with the aid of MAGMA \cite{BC, BCP}, (or by \cite{PT} when $p \geq h - 1$) the theorem
can be verified for types $A_1$, $A_2$, $A_3$, and $B_2$ (or $C_2$) when $p = 3$,  type $A_4$ when $p = 5$, type $A_6$ when $p = 7$, type $F_4$ when $p \geq 7$, and type $G_2$ when $p \geq 5$.  We exclude those cases for the remainder of the proof.
Assume also for the moment that $\Phi$ is not of type $C_n$ ($n \geq 3$) when $p = 5$. 

Let $\ga$ be a weight of $\opH^3(\ul,k)$. Then $\ga= \al + \si_{1}+\si_{2}$ where $\al \in \Pi$ and $\si_{1},\si_{2}\in \Phi^{+}$ by Corollary~\ref{C:h3first}. Now by Proposition~\ref{P:B1-cohostructure}, we have that $\opH^{3}(B_{1},-\gamma)\neq 0$. 
Consider the Lyndon-Hochschild-Serre (LHS) spectral sequence applied to 
$B_{1}\unlhd B$ with $-\gamma+p\nu\in X(T)$:
$$E_{2}^{i,j}=\text{Ext}^{i}_{B/ B_{1}}(-p\nu,\opH^{j}(B_{1},-\ga))
\Rightarrow \opH^{i+j}(B,-\ga+p\nu).$$

First observe that $-\ga\notin pX(T)$ by Proposition~\ref{P:B1-cohostructure}, 
thus $E_{2}^{i,0}=0$ for $i\geq 0$. Next suppose that $E_{2}^{i,1}\neq 0$ for some $i\geq 0$. Then $\opH^{1}(B_{1},-\gamma)\neq 0$.
By \cite{Jan2}, $\gamma=p\sigma+\delta$ for some $\delta\in \Pi$ and $\si \in X(T)$. By Corollary~\ref{C:rootsum-simple}, 
this is not possible.  Thus 
$E_{2}^{i,1}=0$ for $i\geq 0$. Finally, by \cite[Thm 5.3]{BNP2}, if $\opH^{2}(B_{1},-\gamma)\neq 0$,
then $\gamma = -w\cdot 0 + p\sigma$ where $l(w)=2$ and $\si \in X(T)$.  From Proposition~\ref{P:wdotzero},
$-w\cdot 0 = c\de_1 + \de_2$ for $\de_i \in \Pi$ and $c \geq 1$.   By Proposition~\ref{P:rootsum-triple}, we must have
$\si = 0$, and so $\ga = c\de_1 + \de_2$, which is not possible. Therefore, $E_{2}^{i,2}=0$ for $i\geq 2$. 
We can now conclude that 
\begin{equation} 
\opH^{3}(B,-\gamma+p\nu)\cong E_{2}^{0.3} \cong \text{Hom}_{B/B_{1}}(-p\nu,\opH^{3}(B_{1},-\gamma))) \neq 0
\end{equation} 
for some $\nu\in X(T)$. 

From the proof in \cite[2.9]{And} there exists a simple root $\be_1$ such that 
\begin{equation*} 
0\neq \opH^{3}(B,-\gamma+p\nu)\cong \opH^{2}(P_{\be_1},\opH^{1}(P_{\be_1}/B,-\gamma+p\nu))\cong 
\opH^{2}(B,\opH^{1}(P_{\be_1}/B,-\gamma+p\nu)), 
\end{equation*}
where $P_{\be_1} \supset B$ is the standard parabolic subgroup associated to $\beta_1$.  
Therefore, there exists a weight $\mu$ of $\opH^{1}(P_{\beta_1}/B,-\gamma+p\nu))$ with $\opH^{2}(B,\mu)\neq 0$. 
According to \cite[Thm. 5.8(a)]{BNP2}, $\mu=-i_{2}\be_{2}-i_{3}\be_{3}$ for some $\be_2, \be_3 \in \Pi$ where 
$i_{2}, i_{3}\geq 0$ and without loss of generality $i_3 = p^l$ for $l \geq 0$.    
On the other hand, $\mu=-\gamma+p\nu+i_{1}\beta_{1}$ for $i_{1}\geq 0$ by \cite[(2.8)]{And}. Therefore, 
$$\gamma=i_{1}\beta_{1}+i_{2}\beta_{2}+i_{3}\beta_{3}+p\nu.$$ 
If any $i_j \geq p$, the $p$-portion can be combined with the $p\nu$ term to re-express this as
$$\gamma = i_1'\beta_1 + i_2'\be_2 + i_3'\be_3 + p\nu'$$
where $0 \leq i_j' \leq p - 1$ for $j \in \{1,2\}$ and $0 \leq i_3' \leq 1$.  
By Proposition~\ref{P:rootsum-triple}, $\nu' = 0$.  
From Proposition~\ref{P:dotthree} and Proposition~\ref{P:dottwo}, it follows that $\ga=-w\cdot 0$ with $l(w)=3$. 

Lastly, we need to consider the case of type $C_n$ with $p = 5$.  In this setting, Proposition~\ref{P:rootsum-triple},
which has been applied twice, might fail to hold.  Then $\ga$ would need to be 
one of the weights listed in  
Remark~\ref{R:rootsum-triple}.  As noted there, none of those weights have the form $-w\cdot 0$ with $\ell(w) = 3$, 
and they all lie in a type $C_4$ root subsystem.   By direct calculation, the theorem holds for type $C_4$ (or $C_3$).  
In particular, none of the weights from Remark~\ref{R:rootsum-triple} can appear in $\opH^3(\ul,k)$.  Since they do not 
appear in type $C_4$, they also cannot appear as weights of $\opH^3(\ul,k)$ for any $C_n$ for $n \geq 4$.  Hence, the claim holds.
\end{proof} 

\begin{rem}\label{R:u-coho} For odd primes, the conditions on the prime given in Theorem \ref{T:u-coho} are necessary and sufficient to obtain the analogue of Kostant's Theorem in positive characteristic.  Specifically, one can show by direct calculation that there exist ``extra'' cohomology classes in $\opH^3(\ul,k)$ when $p = 3$ in types $A_4$, $D_4$, and $G_2$, .  Such classes will further give rise to ``extra'' cohomology classes in types $A_n$ and $D_n$ when $n > 4$ and types $E_n$.  Similarly, when $p = 3, 5$, there exist ``extra'' cohomology classes in type $B_n$ ($n \geq 4)$ and $F_4$.
\end{rem}


\subsection{\bf Other $u$-cohomology calculations:} In this section, we prove some results about $\opH^i(\ul,\ul^*)$, $i = 0, 1$, that will be used to compute $B_{r}$-cohomology.

\begin{prop} \label{P:Bcoho-ucoho} Let $p$ satisfy Assumption~\ref{A:char}. 
\begin{itemize} 
\item[(a)] $\dim \opH^{0}({\mathfrak u},{\mathfrak u}^{*})_{\la}=
\begin{cases} 
1 & \text{if } \la = \alpha,\; \alpha \in \Pi, \\
0 & \text{else. }
\end{cases}$ 
\item[(b)] Let $\la \in X(T)$.  If $\opH^1(\ul,\ul^*) \neq 0$, then $\la = \al + \be$ where $\al \in \Pi$ and $\be \in \Phi^+$.
\item[(c)] As $T$-modules, $\opH^{1}(U_{1},{\mathfrak u}^{*}) \cong \opH^{1}({\mathfrak u},{\mathfrak u}^{*})$.
\item[(d)] Let $\la \in X(T)$.  If $\lambda\neq \alpha+\beta$ where $\alpha\in \Pi$ and $\beta\in \Phi^{+}$, then 
$$\opH^{1}(B/U_{1},\operatorname{Hom}_{U_{1}}(k,{\mathfrak u}^{*}\otimes -\lambda))\cong 
\opH^{1}(B,{\mathfrak u}^{*}\otimes -\lambda).$$
\item[(e)] If $\lambda=\alpha+\beta$ where $\alpha\in \Pi$ and $\beta\in \Phi^{+}$, then 
$$\opH^{1}({\mathfrak u},{\mathfrak u}^{*})_{\lambda}\cong \opH^{1}(B,{\mathfrak u}^{*}\otimes -\lambda).$$
\end{itemize} 
\end{prop} 

\begin{proof}  Part (a) follows from the fact that with our assumptions on $p$ the $\ul$-socle of ${\mathfrak u}^{*}$ has basis $\{\phi_{\alpha}:\alpha\in \Pi\}$. For part (b), we recall the fact that
$$\opH^{1}({\mathfrak u},{\mathfrak u}^{*})\cong \text{Der}({\mathfrak u},{\mathfrak u}^{*})/\text{Inn}({\mathfrak u},{\mathfrak u}^{*})$$ 
where $\text{Der}({\mathfrak u},{\mathfrak u}^{*})$ (resp. $\text{Inn}({\mathfrak u},{\mathfrak u}^{*})$) are both $T$-modules. 
Since a derivation is completely determined by its values on the generators of ${\mathfrak u}$, if $\opH^{1}({\mathfrak u},{\mathfrak u}^{*})_{\lambda}\neq 0$ then $\lambda=\alpha+\beta$ 
where $\alpha\in \Pi$ and $\beta\in \Phi^{+}$.  

For part (c), there exists a spectral sequence when $p\geq 3$ (cf.~\cite{FP2}, \cite{AJ}, \cite[I.9.20]{Jan1}): 
$$E_{2}^{2i,j}=S^{i}({\mathfrak u}^{*})^{(1)}\otimes \opH^{j}({\mathfrak u},{\mathfrak u}^{*})\Rightarrow \opH^{2i+j}(U_{1},{\mathfrak u}^{*}).$$ 
We show that the differential $d_{2}:\opH^{1}({\mathfrak u},{\mathfrak u}^{*})\rightarrow ({\mathfrak u}^{*})^{(1)}$ is zero. 
As shown above, any weight of $\opH^{1}({\mathfrak u},{\mathfrak u}^{*})$ can be expressed as $\alpha+\beta$ where $\alpha\in \Pi$ and $\beta\in \Phi^{+}$. 
For the differential $d_{2}$ to be non-zero, there would need to be a $\gamma\in \Phi^{+}$ such that 
$\alpha+\beta=p\gamma$. It was shown in \cite[Prop. 3.1(A)]{BNP2} (under conditions weaker than Assumption~\ref{A:char}) that this is not possible if $\al \neq \be$.  However, it is straightforward to see that this is 
also impossible when $\al = \be$, particularly as $\al$ is simple.  
Therefore, $\opH^{1}(U_{1},{\mathfrak u}^{*})\cong \opH^{1}({\mathfrak u},{\mathfrak u}^{*})$. 

For parts (d) and (e), apply the LHS spectral sequence for $U_{1}\unlhd B$ and $\la \in X(T)$: 
\begin{equation} 
E_{2}^{i,j}=\opH^{i}(B/U_{1},\opH^{j}(U_{1},{\mathfrak u}^{*}\otimes -\lambda))\Rightarrow \opH^{i+j}(B,{\mathfrak u}^{*}\otimes -\lambda).
\end{equation} 

First observe that, if $\lambda\neq \alpha+\beta$ where $\alpha\in \Pi$ and $\beta\in \Phi^{+}$, then (using part (b))
$$
E_{2}^{0,1}=\text{Hom}_{B/U_{1}}(\lambda,\opH^{1}(U_{1},{\mathfrak u}^{*}))
\cong \Hom_{B/U_1}(\la,\opH^1(\ul,\ul^*))=0.
$$
Therefore, 
$E_{2}^{1,0} \cong \opH^1(B,\ul\otimes -\la)$ which proves (d). 

Using the fact that $U/U_{1} \unlhd B/U_{1}$ with $(B/U_{1})/(U/U_{1})\cong T$, one has 
\begin{eqnarray*} 
E_{2}^{i,0}&=& \opH^{i}(B/U_{1},\opH^0(U_1,{\mathfrak u}^{*}\otimes -\lambda)) \\
                &\cong & \opH^{i}(U/U_{1},\opH^0(U_1,{\mathfrak u}^{*}\otimes -\lambda))^{T} \\
               &\cong &[\opH^{i}(U/U_{1},k) \otimes \opH^0(U_1,{\mathfrak u}^{*})\otimes -\lambda]^{T} \\
   &\cong &[\opH^{i}(U,k)^{(1)} \otimes \opH^0(\ul,{\mathfrak u}^{*})\otimes -\lambda]^{T}.
\end{eqnarray*} 
By part (a), if $E_{2}^{1,0}\neq 0$, then there exists $\nu\neq 0$ with $p\nu+\gamma-\la=0$ for some 
$\gamma\in \Pi$.   Suppose that $\lambda= \alpha+\beta$ for $\alpha\in \Pi$ and $\beta\in \Phi^{+}$.
Then $\la = \al + \be = \ga + p\nu$. If $\al \neq \be$, this cannot happen under Assumption~\ref{A:char} by \cite[Prop. 3.1(B)]{BNP2}. If $\al = \be \in \Pi$, one can argue in a similar manner that this is impossible.
Therefore, $\opH^1(B,\ul\otimes -\la) \cong E_{2}^{0,1} \cong \Hom_{B/U_1}(k,\opH^1(\ul,\ul^*)\otimes -\la)$ for such $\la$. To finish off the proof, we need to show that $\opH^{1}({\mathfrak u},{\mathfrak u}^{*})$ 
is semisimple as $B/U_{1}$-module. Observe that, for any $\si \in X(T)$,
$$\text{Ext}^{1}_{B/U_{1}}(k,\sigma)\cong \opH^{1}(U/U_{1},\sigma)^{T}\cong [\opH^{1}(U/U_{1},k)\otimes \sigma]^{T}
\cong [\opH^1(U,k)^{(1)}\otimes\si]^{T}.$$
Consequently $\text{Ext}^{1}_{B/U_{1}}(k,\sigma)\neq 0$ implies that $\sigma=-p^{l}\ga$ where $l>0$ and $\ga \in {\mathbb Z}\Phi$. Suppose 
we have two one-dimensional representations in $\opH^{1}({\mathfrak u},{\mathfrak u}^{*})$ which extend one another. 
Viewed as weights, let these be represented by $\de_1 + \be_1$ and $\de_2 + \be_2$ for $\de_1, \de_2 \in \Pi$ and
$\be_1, \be_2 \in \Phi^+$.    Then $\de_{1}+\beta_{1}=\de_{2}+\beta_{2}+p\nu$ with $\nu \in {\mathbb Z}\Phi$ and $\nu\neq 0$. We can rule this out by using Proposition~\ref{P:d2rootsum} (with $\al = \de_1$, $\si_1 = \be_1$, and $\si_2 = 0$).  

\end{proof}

\begin{thm}\label{T:otherucoho} Let $p$ satisfy Assumption~\ref{A:char}. Then 

$$\dim \opH^{1}({\mathfrak u},{\mathfrak u}^{*})_{\la}=
\begin{cases} 
2 & \text{if } \la = \alpha + \beta,\; \alpha, \beta \in \Pi, \al \neq \be, \alpha + \beta \notin \Phi^{+}, \\
1 & \text{if } \la = \alpha + \beta,\; \alpha, \beta \in \Pi, \alpha + \beta \in \Phi^{+}, \\
1 & \text{if } \la = 2\alpha,\; \al \in \Pi,\\
1 & \text{if } \la = -s_{\alpha} s_{ \beta}\cdot 0, \; \alpha, \beta \in \Pi, \alpha + \beta \in \Phi^{+}, \\
0 & \text{else. }
\end{cases}$$ 
\end{thm} 

\begin{proof} We use induction on the rank of the root system. First one can verify the statement directly for root systems of rank less than or equal to two 
(i.e., $\Phi=A_{1}, A_{1}\times A_{1}, A_{2}, B_{2}, G_{2}$) by using the restriction on the allowable weights of $\opH^{1}({\mathfrak u},{\mathfrak u}^{*})$ given in part (b) of Proposition~\ref{P:Bcoho-ucoho}. Alternatively, we can also appeal to \cite[Prop. 4.3]{AR} since our assumptions on $p$ imply $p>h$ for these small rank cases.

The next step is to show that one can reduce from $\ul$ to a subalgebra corresponding to a subset of simple roots $J$ that is properly contained in $\Pi$. Let $J\subseteq \Pi$. Consider the parabolic subalgebra ${\mathfrak p}_J = {\mathfrak l}_J \ltimes {\mathfrak u}_J$
of $\Lie(G)$ where ${\mathfrak l}_J$ is the Levi subalgebra associated to $J$.  
Then one can express ${\mathfrak u}\cong {\mathfrak a}_{J}\ltimes {\mathfrak u}_{J}$. We have the LHS spectral sequence: 
$$E_{2}^{i,j}=\opH^{i}({\mathfrak a}_{J},\opH^{j}({\mathfrak u}_{J},{\mathfrak u}^{*}))\Rightarrow \opH^{i+j}({\mathfrak u},{\mathfrak u}^{*}).$$ 
The five term exact sequence yields 
\begin{equation}
0\rightarrow \opH^{1}({\mathfrak a}_{J},\opH^{0}({\mathfrak u}_{J},{\mathfrak u}^{*}))\hookrightarrow 
\opH^{1}({\mathfrak u},{\mathfrak u}^{*})\rightarrow [\opH^{1}({\mathfrak u}_{J},{\mathfrak u}^{*})]^{{\mathfrak a}_{J}}\rightarrow 
\end{equation} 
Note this is an exact sequence of $T$-modules. If $\lambda\in {\mathbb Z}\Phi_{J}$ then 
\begin{equation}\label{E:prereduce}
\opH^{1}({\mathfrak a}_{J},\opH^{0}({\mathfrak u}_{J},{\mathfrak u}^{*}))_{\lambda} \cong 
\opH^{1}({\mathfrak u},{\mathfrak u}^{*})_{\lambda}
\end{equation} 
because all the weights of $\opH^{1}({\mathfrak u}_{J},{\mathfrak u}^{*})$ are positive linear combinations of simple roots with 
at least one simple root in $\Pi-J$. Now the short exact sequence 
$$0\rightarrow {\mathfrak a}_{J}^{*} \rightarrow {\mathfrak u}^{*} \rightarrow {\mathfrak u}_{J}^{*} \rightarrow 0$$ 
yields a short exact sequence 
$$0\rightarrow \opH^{0}({\mathfrak u}_{J},{\mathfrak a}_{J}^{*})\rightarrow \opH^{0}({\mathfrak u}_{J},{\mathfrak u}^{*}) 
\rightarrow Z \rightarrow 0$$ 
or equivalently, 
$$0\rightarrow {\mathfrak a}_{J}^{*}\rightarrow \opH^{0}({\mathfrak u}_{J},{\mathfrak u}^{*}) 
\rightarrow Z \rightarrow 0,$$ 
where $Z$ is a submodule of $\opH^{0}({\mathfrak u}_{J},{\mathfrak u}_J^{*})$. Observe that any weight in $Z$ is 
not in ${\mathbb Z}\Phi_{J}$. This means that any weight of $\opH^{j}({\mathfrak a}_{J},Z)$ is not in ${\mathbb Z}\Phi_{J}$ 
for $j\geq 0$. Consequently, for $\lambda\in {\mathbb Z}\Phi_{J}$, 
\begin{equation} 
\opH^{1}({\mathfrak a}_{J},{\mathfrak a}_{J}^{*})_{\lambda}\cong 
\opH^{1}({\mathfrak a}_{J},\opH^{0}({\mathfrak u}_{J},{\mathfrak u}^{*}))_{\lambda}.  
\end{equation} 
Combining this with (\ref{E:prereduce}) yields 
\begin{equation} \label{E:reduce}
\opH^{1}({\mathfrak a}_{J},{\mathfrak a}_{J}^{*})_{\lambda}\cong \opH^{1}({\mathfrak u},{\mathfrak u}^{*})_{\lambda}
\end{equation} 
for $\lambda\in {\mathbb Z}\Phi_{J}$. 

As seen in  Proposition \ref{P:Bcoho-ucoho}, every weight of $\opH^{1}({\mathfrak u},{\mathfrak u}^{*})$ is of the form $\alpha+\beta$ where 
$\alpha\in \Pi$ and $\beta\in \Phi^{+}$. Set $\lambda=\alpha+\beta$. If $\lambda\in {\mathbb Z}\Phi_{J}$ where $J$ is properly 
contained in $\Pi$ then one can use the induction hypothesis and (\ref{E:reduce}) to get the claim. 

It remains to show that $\opH^{1}({\mathfrak u},{\mathfrak u}^{*})_{\lambda}=0$ when $\lambda=\alpha+\beta$ where 
$\lambda$ is a positive linear combination with every simple root occurring and  $|\Pi|\geq 3$. If this occurs, by 
Proposition~\ref{P:Bcoho-ucoho}, $\opH^1(B,\ul^*\otimes -\lambda) \neq 0$. 
In \cite[Prop. 4.3]{AR}, a complete computation of $\opH^1(B,\ul^*\otimes -\la)$ is given under the condition 
$p > h$.  That proof consists of two parts: a determination of those $\la$ for which the cohomology is non-zero and then a computation of those cohomology groups. For our purposes, we only need the determination portion of that proof. 
In that portion, it is shown that if $\opH^1(B,\ul^*\otimes -\lambda) \neq 0$, then $\lambda = a\be_1 +b\be_2$ for simple roots $\be_1, \be_2$ and non-negative integers $a, b$, which contradicts our assumption that $\la$ involves at least three simple roots, and hence completes the  proof.   The necessary portion of the argument in \cite{AR} only requires $p \geq 5$ (in order to guarantee that $\langle \eta + \delta,\ga^{\vee}\rangle < p$ for simple roots $\eta, \ga$ and a positive root $\de$). In particular, it holds under our Assumption \ref{A:char}.   
\end{proof} 

\begin{rem} An alternative,non-inductive, proof of the aforementioned result can also be obtained by arguing along the lines used in \cite{BNP2} to compute $\opH^2(\ul,k)$, and again using \cite[Prop. 4.3]{AR}.
\end{rem}


\section{$B_r$-cohomology}


\subsection{} In this section, we compute $\opH^3(B_r,\la)$ for 
all $\la \in X(T)$.  We recall that the first cohomology groups 
$\opH^1(B_1,\la)$ were computed for all primes and all weights $\la \in X(T)$ by Jantzen \cite{Jan2}.
For higher $r$, $\opH^1(B_r,\la)$ is computed by the authors in \cite{BNP1},
and $\opH^2(B_r,\la)$ is computed by the authors \cite{BNP2} and Wright \cite{W}.

We investigate the third cohomology $\opH^3(B_r,\la)$ by starting out with $B_1$.
Note that for $\la \in X(T)$, we may write $\la = \la_0 + p\la_1$
for unique weights $\la_0, \la_1$ with $\la_0 \in X_1(T)$.  Then 
\begin{equation}
\opH^3(B_1,\la) = \opH^3(B_1,\la_0 + p\la_1) \cong 
        \opH^3(B_1,\la_0)\otimes p\la_1.
\end{equation} 
Hence, it suffices to compute $\opH^3(B_1,\la)$ for $\la \in X_1(T)$. To this end, we will be interested
in $\la$ of the form $\lambda=w\cdot 0 + p\nu \in X_1(T)$ with $w\in W$ and $l(w)=3$.
Given such a $w$, there exists a unique weight $\nu \in X(T)$ such that $\la = w\cdot 0 + p\nu \in X_1(T)$.  
Such weights $\nu$ are summarized in Lemma \ref{L:gammaw} (in the Appendix).



\subsection{} We can now present the computation of $\opH^{3}(B_{1},\lambda)$ for $\lambda\in X_{1}(T)$. 
The weight $\ga_w$ given in the statement of the theorem is identified in Lemma \ref{L:gammaw}.  The precise
identification of $\ga_w$ is not necessary for the proof.

\begin{thm}\label{T:B1coho} Let $p$ satisfy Assumption~\ref{A:char} and $\la \in X_1(T)$. Then as $B/B_{1}$-modules
$$
\opH^{3}(B_{1},\lambda)\cong 
\begin{cases} 
\gamma_{w}^{(1)} &\text{if } \lambda=w\cdot 0+p\ga_{w} \text{ with } \ell(w)=3, \\
({\mathfrak u}^{*})^{(1)}\otimes \omega_{\alpha}^{(1)}      & \text{if } \lambda=s_{\al}\cdot 0+p\omega_{\al} \text{ where }
\al\in \Pi, \\
0    & \text{otherwise}.
\end{cases} 
$$
\end{thm} 

\begin{proof} From Proposition~\ref{T:utoU1}, we have the following isomorphisms as $T/T_{1}$-modules: 
\begin{eqnarray*} 
\opH^{3}(B_{1},\lambda) &\cong& [\opH^{3}(U_{1},k)\otimes \lambda]^{T_{1}}\\
&\cong & [ \opH^{3}({\mathfrak u},k)\otimes \lambda]^{T_{1}}\oplus \left(({\mathfrak u}^{*})^{(1)}\otimes 
[\opH^{1}({\mathfrak u},k)\otimes \lambda]^{T_{1}}\right).
\end{eqnarray*}
Note that the first isomorphism holds as $B/B_1$-modules and the second isomorphism is obtained from a spectral 
sequence on which $B/B_1$ acts, preserving the differentials.   We will argue below that these two summands 
cannot occur simultaneously, and hence the identification of $\opH^3(B_1,\la)$ holds as a $B/B_1$-module.

Therefore, if $p\nu$ is a weight of $\opH^{3}(B_{1},\lambda)$, then,
by Theorem~\ref{T:u-coho} and Proposition~\ref{P:h1},
either (i) $p\nu=-w\cdot 0 + \lambda$ for some $w\in W$ with $l(w)=3$ or 
(ii) $p\nu= p\si + \be + \la$ with $\si \in \Phi^+$, $\be \in \Pi$. 
In the first case, since $\lambda \in X_{1}(T)$, $\nu = \ga_{w}$.   
In the second case, $\lambda=-\be+p(\nu - \si) = s_{\be}\cdot 0 + p(\nu - \si)$ and we must have
$\nu - \si = \omega_{\be}$. 

Given $\la \in X_1(T)$, we consider whether it can simultaneously take form (i) and (ii). This would imply that
$w\cdot 0 + p\ga_{w} = -\be + p\omega_{\be}$ where $\ell(w) = 3$ and $\be \in \Pi$.  Equivalently, we would have
$-w\cdot 0 = \be + p(\ga_{w} - \omega_{\be})$.   By Corollary~\ref{C:rootsum-simple} (with $-w\cdot 0$
playing the role of $\al + \si_1 + \si_2$), this is impossible.  

Lastly, it remains to check whether the weights of type (i) or (ii) can occur in multiple ways, something that would 
lead to doubling of the cohomological dimension.
For $p \geq 5$, it is straightforward to check (or see \cite{BNP1}) that if $-\be_1 + p\omega_{\be_1} = -\be_2 + p\omega_{\be_2}$,
then $\be_1 = \be_2$.   Similarly, suppose that $w_{1}\cdot 0+p\nu_{1}=w_{2}\cdot 0+p\nu_{2}\in X_{1}(T)$ for some $w_1$, $w_2$ 
with $l(w_{1})=l(w_{2})=3$.  Such an equation can be rewritten as $-w_1\cdot 0 = -w_2\cdot 0 + p(\nu_1 - \nu_2)$. 
We may apply Proposition~\ref{P:rootsum-triple} 
(with $-w_1\cdot 0$ playing the role of $\al + \si_1 + \si_2$ and $-w_2 \cdot 0$ playing the role of $i_1\be_1 + i_2\be_2 + \be_3$)
to conclude that $\nu_1 - \nu_2 = 0$ or $\nu_1 = \nu_2$, from which it follows that $w_1 = w_2$.

This application of Proposition~\ref{P:rootsum-triple} potentially fails in types $C_n$ for $p = 5$, $F_4$ for $p = 7$,
and $G_2$ for $p = 7$.  However, from Remark~\ref{R:rootsum-triple}, none of the potential ``bad'' weights have the 
form $-w\cdot 0$ for $\ell(w) = 3$, and so the claim follows.
\end{proof}

\begin{rem}  
For type $A_4$, when $p = 5$, and type $A_6$, when $p=7$, the structure of $\opH^3(B_1,\la)$ may be more complex. In these cases there exist pairs $w_1, w_2 \in W,$ both of length at most $3$, with $w_1\cdot 0 = w_2 \cdot 0 + p\nu$ for some $\nu \in X(T).$ Following the 
discussion in \cite[6.1 - 6.3]{AJ} one observes that this happens if and only if the pair $w_1, w_2$ appears in the same left coset of the cyclic subgroup of $W$ generated by $s_1s_2 ... s_n.$ 
For example, in type $A_4$, let $w_1= s_4$ and 
$w_2 = s_3s_2s_1$ then $\la = s_4\cdot 0 + 5\omega_4= s_3s_2s_1\cdot 0 + 5\omega_3$. The weight $\la$ takes both 
forms as given in the theorem and the two cases combine. 
For type $A_6$ there is just the pair, $w_1=s_3s_2s_1, w_2 = s_4s_5s_6.$  Note that these  are the ``bad" cases referred to in Proposition~\ref{P:rootsum-triple}, see also Remark~\ref{R:rootsum-simple}. 
From length considerations one can see that no such pairs occur in ranks greater than $6$ and cohomological degree 3.
\end{rem}


\subsection{} The preceding calculations can be used to compute
$\opH^3(B_r,\la)$ for any $r$ and $\lambda\in X(T)$.  We first need to review the known computations of 
$\opH^{j}(B_{1},\lambda_{0})$ for $j=0,1,2$ as a $B/B_{1}$-module where $\lambda_{0}\in X_{1}(T)$ 
when $p$ satisfies Assumption~\ref{A:char}:
\begin{equation} \label{E:H0}
\opH^{0}(B_{1},\lambda_{0})\cong \begin{cases} k & \text{if } \lambda_{0}=0, \\
0 & \text{else},
\end{cases}
\end{equation}
\begin{equation} \label{E:H1}
\opH^{1}(B_{1},\lambda_{0})\cong \begin{cases} \omega_{\alpha}^{(1)} & \text{if } \lambda_{0}=-\alpha+p\omega_{\alpha}, \alpha\in \Pi, \\
0 & \text{else},
\end{cases}
\end{equation}
\begin{equation} \label{E:H2}
\opH^{2}(B_{1},\lambda_{0})\cong \begin{cases} ({\mathfrak u}^{*})^{(1)} & \text{if } \lambda_{0}=0, \\
\gamma_{w}^{(1)}  & \text{if } \lambda_{0}=w\cdot 0+p\gamma_{w}, l(w)=2, \\

0 & \text{else}.
\end{cases}
\end{equation}
We also note that more generally, for $\la \in X(T)$,
\begin{equation}\label{E:H0forr}
\opH^0(B_r,\la) \cong
\begin{cases} \nu^{(r)} & \text{if } \la = p^r\nu,\\
0 & \text{else}.
\end{cases}
\end{equation}
Lastly, we need the following observation.

\begin{lem}\label{L:4,0}
Let $p$ satisfy Assumption~\ref{A:char} and  $\al, \be \in \Pi$.
\begin{itemize}
\item[(a)]
If $\la = -\al,$ then the zero weight space of $\opH^3(B_r, \la)$ is zero. 
\item[(b)]
 If $\la \in \{ -\al-p^i\be \; | \; 0\leq i \leq r-1\},$ then the zero weight space of $\opH^4(B_r, \la)$ is zero. 
 \item[(c)]
 If $\la = s_{\al}s_{\be} \cdot 0$ and $\al+ \be \in \Phi^+,$ then the zero weight space of $\opH^4(B_r, \la)$ is zero. 
 \end{itemize}
\end{lem}

\begin{proof} Recall that $\opH^i(B_r,\la) \cong \opH^i(U_r,\la)^{T_r}$ and consider the spectral sequence of \cite[I.9.14]{Jan1} abutting to $\opH^{\bullet}(U_r,\la)$ (whose differentials preserve the action of $T$).  The only terms that can contribute to $\opH^3(U_r,\la)$ have the form $\la\otimes(\ul^*)^{(i)}\otimes(\ul^*)^{(j)}$ or $\la\otimes\Lambda^3(\ul^*)^{(j)}$ for $1 \leq i \leq r$ and $0 \leq j \leq r-1$.  For $\la = -\al$, neither term contains the zero weight.  Similarly the only terms that contribute to $\opH^4(U_r,\la)$ are $\la\otimes S^2(\ul^*)^{(i)}$, 
$\la\otimes(\ul^*)^{(i_1)}\otimes(\ul^*)^{(i_2)}$, $\la\otimes(\ul^*)^{(i)}\otimes\Lambda^2(\ul^*)^{(j)}$, $\la\otimes(\ul^*)^{(i)}\otimes(\ul^*)^{(j_1)}\otimes(\ul^*)^{(j_2)}$, or $\la\otimes\Lambda^4(\ul^*)^{(j)}$ for $1 \leq i, i_1, i_2 \leq r$, $0 \leq j \leq r - 1$, and $0 \leq j_1 < j_2 \leq r-1$.
\end{proof}

The following proposition provides a recursive algorithm to compute $\opH^3(B_r,\la)$. 

\begin{prop}\label{Br:prop} Let $p$ satisfy Assumption~\ref{A:char}, $r \geq 2$, and $\lambda \in X(T)$. 
\begin{itemize} 
\item[(a)] If $\lambda_{0}\neq 0$, then 
\begin{equation*}
\opH^{3}(B_{r},\lambda)\cong \begin{cases} 
\nu^{(r)} &\text{if } \lambda=p^{r}\nu-p^{i}\alpha+ w\cdot 0, \;l(w)=2,\\ 
                                    &\quad 1\leq i\leq r-1, \al \in \Pi, \nu \in X(T),\\
    \nu^{(r)} &\text{if } \lambda=p^{r}\nu+p^{i}w\cdot 0- \alpha , \;l(w)=2,\\ 
                                    &\quad 1\leq i\leq r-1, \al \in \Pi, \nu \in X(T),\\                                
\nu^{(r)} & \text{if } \lambda=p^{r}\nu+w\cdot 0, \; l(w)=3, \nu \in X(T),\\
                 
\nu^{(r)} & \text{if } \lambda=p^{r}\nu-p^i\beta-\alpha,\\
&\quad  1\leq i \leq r-1, \alpha,\beta\in \Pi, \nu \in X(T),\\
 \nu^{(r)} & \text{if } \lambda = p^r\nu -p^l \gamma - p^i \beta - \alpha,\\
&\quad 1 \leq i < l \leq r-1, \alpha, \beta, \gamma \in \Pi, \nu \in X(T),
 \\

(\mathfrak{u}^{*} \otimes \nu)^{(r)}& \text{if } \la = p^r\nu - \alpha, \alpha \in \Pi, \nu \in X(T),\\

0 &\text{otherwise}.
\end{cases} 
\end{equation*} 

\item[(b)] If $\lambda_{0}=0$, then 
\begin{equation*}
\opH^{3}(B_{r},\lambda)\cong 
\begin{cases}
 
{\nu}^{(r)}\oplus {\nu}^{(r)} & \text{if } \lambda =p^{r}\nu-p^i\beta-p\alpha,  \\
     &\quad 2 \leq i \leq r-1, \alpha,\beta\in \Pi, \nu \in X(T),\\
   
   {\nu}^{(r)}\oplus {\nu}^{(r)} & \text{if } \lambda =p^{r}\nu-p\beta-p\alpha,  \\
     &\quad  \alpha,\beta\in \Pi, \al \neq \be, \al + \be \notin \Phi^+, \nu \in X(T),\\
     
       {\nu}^{(r)}& \text{if } \lambda =p^{r}\nu-p\beta-p\alpha,  \\
     &\quad  \alpha,\beta\in \Pi, \al + \be \in \Phi^+, \nu \in X(T),\\
     
     {\nu}^{(r)}& \text{if } \lambda =p^{r}\nu-2p\alpha,\; \alpha \in \Pi, \nu \in X(T),\\
     
        {\nu}^{(r)}& \text{if } \lambda =p^{r}\nu+ps_{\al}s_{\be} \cdot 0,  \\
     &\quad  \alpha,\beta\in \Pi, \al + \be \in \Phi^+, \nu \in X(T),\\

\opH^{3}(B_{r-1},\lambda_{1})^{(1)} & \text{otherwise}. 
\end{cases}
\end{equation*} 
\end{itemize}
\end{prop} 

\begin{proof} Consider the LHS spectral sequence 
$$E_{2}^{i,j}=\opH^{i}(B_{r}/B_{1},\opH^{j}(B_{1},\lambda_{0})\otimes p\lambda_{1})\Rightarrow \opH^{i+j}(B_{r},\lambda)$$ 
where $\lambda=\lambda_{0}+p\lambda_{1}$, $\lambda_{0}\in X_{1}(T)$ and $\lambda_{1}\in X(T)$. 
The possible terms that can contribute to $\opH^{3}(B_{r},\lambda)$ are $E_{2}^{i,j}$ where $i+j=3$,
and this can be non-zero only if $\opH^j(B_{1},\la_0) \neq 0$ for some $j = 0, 1, 2, 3$.  These cases 
will be analyzed using Theorem~\ref{T:B1coho} and Equations 
(\ref{E:H0})-(\ref{E:H2}).  Note that, by Proposition~\ref{P:d4rootsum}, Proposition~\ref{P:d2rootsum}, Corollary~\ref{C:rootsum-simple}, and \cite[Prop. 3.1(A)(B)]{BNP2}, the four possibilities given below for $\la_0$ cannot happen simultaneously.

\vskip .25cm 
\noindent
{\it Case 1}: $\lambda_{0}=w\cdot 0+p\ga_{w}$, $l(w)=2$ 
\vskip .25cm 

In this case, we have $\opH^{j}(B_{1},\lambda_{0})=0$ for $j=0,1,3$ and $\opH^2(B_1,\la_0) \cong \ga_w^{(1)}$. Therefore, 
$$\opH^{3}(B_{r},\lambda)\cong E_{2}^{1,2}\cong 
\opH^{1}(B_{r}/B_{1},p(\ga_{w}+\lambda_{1}))\cong  \opH^{1}(B_{r-1},\ga_{w}+\lambda_{1})^{(1)}.$$ 
According to \cite[Thm. 2.8(A)]{BNP1},  $\opH^{1}(B_{r-1},\ga_{w}+\lambda_{1})^{(1)}\cong \nu^{(r)}$ if 
$\ga_{w}+\lambda_{1}=p^{r-1}\nu-p^{i}\alpha$ for some $\alpha\in \Pi$, $0\leq i \leq r-2$, otherwise it is zero. 
This translates to  
$\lambda=w\cdot 0+p^{r}\nu-p^{i}\alpha$ where $\alpha\in \Pi$ and $1\leq i\leq r-1$. 

\vskip .25cm 
\noindent 
{\it Case 2}: $\lambda_{0}=w\cdot 0+p\gamma_{w}$, $l(w)=3$
\vskip .25cm 

This condition implies that $\opH^{j}(B_{1},\lambda_{0})=0$ for $j=0,1,2$. Thus, 
$$\opH^{3}(B_{r},\lambda)\cong E_{2}^{0,3}= \opH^0(B_r/B_1,\opH^3(B_1,w\cdot 0 + p\ga_w)\otimes p\la_1) \cong \opH^{0}(B_{r-1},\gamma_{w}+\lambda_{1})^{(1)}$$ 
where the last isomorphism follows from Theorem~\ref{T:B1coho}.
Moreover, 
$$\opH^{0}(B_{r-1},\gamma_{w}+\lambda_{1})\cong \begin{cases}
\nu^{(r-1)} & \text{if } \gamma_{w}+\lambda_{1}=p^{r-1}\nu, \nu \in X(T), \\
0  & \text{if } \gamma_{w}+\lambda_{1}\notin p^{r-1}X(T). 
\end{cases} 
$$ 
The non-vanishing case translates to $\lambda = w\cdot 0 + p^r\nu$. 

\vskip .25cm 
\noindent 
{\it Case 3}: $\lambda_{0}=-\alpha+p\omega_{\alpha}$ 
\vskip .25cm 
We have $\opH^{j}(B_{1},\lambda_{0})=0$ for $j=0,2$ and $E_{2}^{2,1}\hookrightarrow \opH^{3}(B_{r},\lambda)$. 
In this case we have a possible non-zero differential at the $E_{3}$-level: $d_{3}:E_{3}^{0,3}\rightarrow E_{3}^{3,1}$. 
Note that $E_{2}^{0,3}=E_{3}^{0,3}$ and $E_{2}^{3,1}=E_{3}^{3,1}$. One can immediately conclude that (as $T/T_r$-modules)
$\opH^{3}(B_{r},\lambda)\cong E_{2}^{2,1}\oplus \text{ker }d_{3}$. One of the goals will be to show that $d_{3}=0$ so that 
\begin{equation}\label{E:case3}
\opH^{3}(B_{r},\lambda)\cong E_{2}^{2,1}\oplus E_{2}^{0,3}.
\end{equation}
In the following, we apply Theorem~\ref{T:B1coho} and Proposition~\ref{P:Bcoho-ucoho}(a) to obtain
\begin{eqnarray*} 
E_{2}^{0,3}&\cong & \opH^{0}(B_{r}/B_{1},\opH^{3}(B_{1},\lambda_{0})\otimes p\lambda_{1}) \\
&\cong & \opH^{0}(B_{r}/B_{1}, ({\mathfrak u}^{*})^{(1)}\otimes p(\omega_{\alpha}+\lambda_{1}))\\
&\cong & \opH^{0}(B_{r-1}, {\mathfrak u}^{*} \otimes(\omega_{\alpha}+\lambda_{1}))^{(1)} \\
&\cong & [(\opH^0(U_{r-1},\ul^*)\otimes(\omega_{\al} + \la_1))^{T_{r-1}}]^{(1)}\\
&\cong & [(\opH^0(\ul,\ul^*)\otimes(\omega_{\al} + \la_1))^{T_{r-1}}]^{(1)}\\
&\cong & \begin {cases} \nu^{(r)} & \text{if } \lambda_{1}= p^{r-1}\nu- \beta-\omega_{\alpha}, \text{ for } \beta\in \Pi, \nu \in X(T), \\
0  & \text{else } 
\end{cases}\\
&\cong & \begin {cases} \nu^{(r)} & \text{if } \lambda= p^{r}\nu- p\beta-\alpha, \text{ for } \alpha, \beta\in \Pi, \nu \in X(T), \\
0  & \text{else}. 
\end{cases} 
 \end{eqnarray*} 
On the other hand, 
\begin{equation*} 
E_{2}^{3,1}\cong \opH^{3}(B_{r}/B_{1},\opH^{1}(B_{1},\lambda_{0})\otimes p\lambda_{1}) 
\cong \opH^{3}(B_{r-1},\omega_{\alpha}+\lambda_{1})^{(1)}.
\end{equation*} 
Now suppose that $E_{2}^{0,3}\neq 0$. Then $\omega_{\alpha}+\lambda_{1}= p^{r-1}\nu-\beta$ for some 
$\beta\in \Pi, \nu \in X(T)$. It follows that 
$$E_{2}^{3,1}\cong \opH^3(B_{r-1},\omega_{\al} + \la_1)^{(1)} \cong \opH^{3}(B_{r-1},-\beta)^{(1)}\otimes \nu^{(r)}.$$ 
By Lemma~\ref{L:4,0}, $\opH^{3}(B_{r-1},-\beta)^{(1)}$ does not have a zero weight space. Since the differential $d_{3}$ must preserve 
$T/T_{r}$ weight spaces, it follows that $d_{3}=0$. Consequently, (\ref{E:case3}) holds. 

If $E_{2}^{0,3}\neq 0$, then $\beta+\omega_{\alpha}+\lambda_{1}= p^{r-1}\nu$ for some $\beta\in \Pi, \nu \in X(T)$. 
When this occurs 
$$E_{2}^{2,1}= \opH^2(B_r/B_1,p(\omega_{\al} + \la_1)) \cong \opH^{2}(B_{r-1},-\beta)^{(1)}\otimes \nu^{(r)}=0$$ 
by \cite[Lemma 5.6]{BNP2}. Hence, under the given assumption on $\la_0$, if $\opH^3(B_r,\la) \neq 0$, then 
$\opH^3(B_r,\la) \cong E_2^{0,3}$ or $\opH^3(B_r,\la) \cong E_2^{2,1} \cong \opH^2(B_{r-1},\omega_{\al} + \la_1)^{(1)}.$ 

In the latter case one obtains from \cite[Thm. 5.7]{BNP2} that 
\begin{equation*}
\opH^{3}(B_{r},\lambda)\cong \begin{cases} 
(\mathfrak{u}^{*} \otimes \nu)^{(r)}& \text{if } \la = p^r\nu - \alpha, \alpha \in \Pi, \nu \in X(T),\\

\nu^{(r)} & \text{if } \lambda = p^r\nu + p^i w \cdot 0- \alpha, \ell(w) =2,\\
&\quad 1 \leq i \leq r-1, \alpha \in \Pi, \nu \in X(T),
 \\

\nu^{(r)} & \text{if } \lambda = p^r\nu - p^i \beta - \alpha,\\
&\quad 2 \leq i \leq r-1, \alpha, \beta \in \Pi, \nu \in X(T),
 \\
 
 \nu^{(r)} & \text{if } \lambda = p^r\nu -p^l \gamma - p^i \beta - \alpha,\\
&\quad 1 \leq i < l \leq r-1, \alpha, \beta, \gamma \in \Pi, \nu \in X(T),
 \\
0 & \text{otherwise}. 
\end{cases}
\end{equation*} 
Combining these gives the remaining conditions on $\la$ for non-vanishing in part (a) of the theorem and completes the proof of the theorem for the case $\la_0 \neq 0$.

\vskip .25cm 
\noindent 
{\it Case 4}: $\lambda_{0}=0$ 
\vskip .25cm 
In this case, we have $\opH^{j}(B_{1},\lambda_{0})=0$ for $j=1,3$. We have 
\begin{equation*} 
E_{3}^{1,2}=E_{2}^{1,2} \cong \opH^{1}(B_{r-1},{\mathfrak u}^{*}\otimes \lambda_{1})^{(1)}
\end{equation*} 
and 
\begin{equation*} 
E_{3}^{3,0}=E_{2}^{3,0} \cong \opH^{3}(B_{r-1}, \lambda_{1})^{(1)} .
\end{equation*} 
We will show that (as $T/T_r$-modules)
$$\opH^{3}(B_{r},\lambda)\cong E_{3}^{3,0} \oplus E_{3}^{1,2}.$$ 
To do so, we will show that the differentials $d_3 : E_3^{0,2} \to E_3^{3,0}$ 
and $d_{3}:E_{3}^{1,2}\rightarrow E_{3}^{4,0}$ are both zero. 

In the first case, we have
$$
E_3^{0,2} = \opH^0(B_r/B_1,\opH^2(B_1,k)\otimes p\la_1) \cong \opH^0(B_{r-1},\ul^*\otimes \la_1)^{(1)}.
$$
As seen in Case 3, this is non-zero only if $\la_1 + \be = p^{r-1}\nu$ for some $\be \in \Pi$ and $\nu \in X(T)$.
In that case, $E_3^{0,2} \cong \nu^{(r)}$ and we have
$$
E_3^{3,0} \cong \opH^3(B_{r-1},-\be + p^{r-1}\nu)^{(1)} \cong \opH^3(B_{r-1},-\be)^{(1)}\otimes\nu^{(r)}.
$$
Lemma~\ref{L:4,0} now implies that $\opH^3(B_{r-1},-\be)$ has no zero weight space. Hence, the differential is the zero map.

To consider the second differential, we first need to compute $E_{3}^{1,2}$. 
For $r=1$, Theorem~\ref{T:otherucoho} can be used to see that 

\begin{align}\label{E:H1B1}
\opH^{1}(B_{1},{\mathfrak u}^{*}\otimes \lambda_{1})&\cong [\opH^{1}(U_{1},{\mathfrak u}^{*})\otimes \lambda_{1}]^{T_{1}}\notag\\
 &\cong [\opH^{1}({\mathfrak u},{\mathfrak u}^{*})\otimes \lambda_{1}]^{T_{1}}\notag\\
 &\cong 
	\begin{cases}
		\nu^{(1)} \oplus \nu^{(1)} & \text{if }\lambda_{1}=p\nu-\alpha -\beta, \text{ for }  \alpha, \beta \in \Pi,\\
			&\quad \al \neq \be,\; \alpha+\beta \notin \Phi^+, \; \nu \in X(T),\\	
		\nu^{(1)}  & \text{if }\lambda_{1}=p\nu-\alpha -\beta, \text{ for }  \alpha, \beta \in \Pi,\\
			&\quad \alpha+\beta \in \Phi^+, \; \nu \in X(T), \\
		\nu^{(1)} & \text{if } \lambda_1 = p\nu - 2\al, \text{ for } \al \in \Pi,\; \nu \in X(T),\\
		\nu^{(1)} & \text{if }\lambda_{1}=p\nu +s_{\alpha}s_{\beta}\cdot 0, \text{ for }  \alpha, \beta \in \Pi,\\
			&\quad \alpha+\beta \in \Phi^+, \; \nu \in X(T), \\
		0 & \text{else}.
	\end{cases}
 \end{align}
Similarly, one obtains
\begin{eqnarray}\label{E:H0B1}
\opH^{0}(B_{1},{\mathfrak u}^{*}\otimes \lambda_{1})&\cong 
& \begin{cases} 
\nu^{(1)} & \text{if }\lambda_{1}=p\nu-\alpha, \text{ for }  \alpha \in \Pi, \; \nu \in X(T), \\
0  & \text{else}. 
\end{cases}
\end{eqnarray} 
One can apply the LHS spectral sequence (for $B_1 \unlhd B_{r-1}$) abutting to $\opH^{\bullet}(B_{r-1},\ul^*\otimes \la_1)$
to obtain an exact sequence 
$$
0 \to E_1^{1,0} \to E^1 \to E_2^{0,1} \to E_2^{2,0}
$$
which is equivalent to 
\begin{eqnarray*}
0 &\rightarrow& \opH^{1}(B_{r-2},\opH^{0}(B_{1},{\mathfrak u}^{*}\otimes \lambda_{1})^{(-1)})^{(1)} 
\rightarrow \opH^{1}(B_{r-1},{\mathfrak u}^{*}\otimes \lambda_{1})\\ 
&\rightarrow & \opH^{0}(B_{r-2},\opH^{1}(B_{1},{\mathfrak u}^{*}\otimes \lambda_{1})^{(-1)})^{(1)} 
\rightarrow \opH^{2}(B_{r-2},\opH^{0}(B_{1},{\mathfrak u}^{*}\otimes \lambda_{1})^{(-1)})^{(1)}.\notag
\end{eqnarray*} 
Note that Assumption~\ref{A:char}  does not allow for both $\opH^{0}(B_{1},{\mathfrak u}^{*}\otimes \lambda_{1})$ and $\opH^{1}(B_{1},{\mathfrak u}^{*}\otimes \lambda_{1})$ to be non-zero.  

First we consider the case  
$\opH^{0}(B_{1},{\mathfrak u}^{*}\otimes \lambda_{1})\neq 0.$
By (\ref{E:H0B1}),
$\lambda_{1}= p\kappa-\alpha$ for some $\alpha \in \Pi$ and $\kappa \in X(T)$. Since  $\opH^{1}(B_{1},{\mathfrak u}^{*}\otimes \lambda_{1})= 0$ one concludes
\begin{equation*} 
\opH^{1}(B_{r-1},{\mathfrak u}^{*}\otimes \lambda_{1})\cong \opH^{1}(B_{r-2},\opH^{0}(B_{1},{\mathfrak u}^{*}\otimes \lambda_{1})^{(-1)})^{(1)} 
\cong \opH^{1}(B_{r-2},\kappa)^{(1)}.  
\end{equation*} 
Generalizing (\ref{E:H1}), from \cite[Thm. 2.8(A)]{BNP1}, we have that  $\opH^{1}(B_{r-2},\kappa)^{(1)}\neq 0$ if and only if $\kappa=p^{r-2}\nu-p^{i}\beta$ for some 
$\beta \in \Pi$ with $0\leq i \leq r-3$. This translates to $\la_1 = p^{r-1}\nu- p^{i+1}\beta -\al$, in which case we have $\la = p^{r}\nu- p^{i}\beta -p\al$ with $2\leq i \leq r-1$ and
$
E_3^{1,2} \cong  {\nu}^{(r)}.
$

Next we consider the case  
$\opH^{0}(B_{1},{\mathfrak u}^{*}\otimes \lambda_{1})= 0$ and
$\opH^1(B_1,\ul\otimes\la_1) \neq 0$.  
By (\ref{E:H1B1}),
$\lambda_{1}= p\kappa-\alpha-\beta$, $\la_1 = p\kappa - 2\al$, or $\lambda_{1}= p\kappa+s_{\alpha}s_{\beta}\cdot 0$ for some $\alpha,\beta \in \Pi$ and $\kappa \in X(T).$ Moreover, from the above exact sequence, 
$$\opH^{1}(B_{r-1},{\mathfrak u}^{*}\otimes \lambda_{1})^{(1)}\cong \opH^{0}(B_{r-2},\opH^{1}(B_{1},{\mathfrak u}^{*}\otimes \lambda_{1})^{(-1)})^{(1)}.$$ 
For the latter to be non-zero, by (\ref{E:H1B1}) and (\ref{E:H0forr}), we need $\kappa = p^{r-2}\nu$ for some $\nu \in X(T)$. Combining these with our earlier observation we conclude that
\begin{align}\label{E:1,2}
E_3^{1,2} &\cong \opH^1(B_{r-1},\ul^*\otimes\la_1)^{(1)}\notag\\
	&\cong
\begin{cases}
\nu^{(r)}   & \text{if } \lambda_{1}=p^{r-1}\nu-p^i\beta-\alpha  \text{ for } \al, \beta \in \Pi, \al + \be \in \Phi^+,\\
	&\quad  1 \leq i \leq r-2, \nu \in X(T), \\
 \nu^{(r)} \oplus  \nu^{(r)}  & \text{if } \lambda_{1}=p^{r-1}\nu-\alpha-\beta  \text{ for } \al, \beta \in \Pi, \al \neq \be,\\
	&\quad \al + \be \notin \Phi^+, \nu \in X(T), \\
\nu^{(r)}   & \text{if } \lambda_{1}=p^{r-1}\nu-\alpha-\beta  \text{ for } \al, \beta \in \Pi, \al + \be \in \Phi^+,\\
	&\quad  \nu \in X(T), \\
\nu^{(r)} & \text{if } \la_1 = p^{r-1}\nu - 2\al \text{ for } \al \in \Pi,\; \nu \in X(T),\\	
\nu^{(r)}   & \text{if } \lambda_{1}=p^{r-1}\nu+s_{\alpha}s_{\beta} \cdot 0  \text{ for } \al, \beta \in \Pi, \al + \be \in \Phi^+,\\
	&\quad  \nu \in X(T), \\
0 & \text{ otherwise}.
\end{cases}
\end{align}
Now it follows from Lemma~\ref{L:4,0} that at least one of the terms $E_3^{1,2}$ and $E_3^{4,0}$ vanishes.
Hence $d_3:E_3^{1,2} \to E_3^{4,0}$ is the zero map and $\opH^3(B_r, \la) = E_3^{1,2} \oplus E_3^{3,0}. $

Finally, we consider whether $E_3^{1,2}$ and $E_3^{3,0}$ can simultaneously be non-zero.  Suppose that $E_3^{1,2} \neq 0
$. Then $\la = p\la_1$ for one of the weights given in (\ref{E:1,2}).   Note that in all (non-zero) cases $\la_1 \notin pX(T)$.  On the other hand, for $\la = p\la_1$, $E_3^{3,0} \cong \opH^3(B_{r-1},\la_1)^{(1)}$.  Inductively,
applying the theorem to $r - 1$, $E_3^{3,0}$ can be non-zero only for those weights listed in part (a) of the theorem.   Comparing lists, we see that $E_3^{1,2}$ and $E_3^{3,0}$ are both non-zero if and only if $\la_1 = p^{r}\nu- p^{i}\beta -\al$ with $1\leq i \leq r-2$ (necessarily requiring $r \geq 3$).  In this case 
$E_3^{3,0} = E_3^{1,2}= {\nu}^{(r)}.$
The statement of the theorem in the $\la_0 = 0$ case follows from our above analysis. 
\end{proof} 


We now state the results of Proposition~\ref{Br:prop} in closed form.

\begin{thm}\label{T:Brfinal} Let $p$ satisfy Assumption~\ref{A:char}, $r \geq 2$, and  $\lambda \in X(T)$. 
\begin{equation*}
\opH^{3}(B_{r},\lambda)\cong \begin{cases} 

(\mathfrak{u}^{*} \otimes \nu)^{(r)}& \text{if } \la = p^r\nu - p^l\alpha, \; \\
&\quad 0\leq l \leq r-1,\; \alpha \in \Pi, \; \nu \in X(T),\\

\nu^{(r)} & \text{if } \lambda=p^{r}\nu+p^l w\cdot 0, \; l(w)=3, \\
&\quad 0\leq l \leq r-1,\; \nu \in X(T),\\

\nu^{(r)} &\text{if } \lambda=p^{r}\nu-p^{m}\alpha+ p^l w\cdot 0, \;l(w)=2,\\ 
                                    &\quad 0\leq l <  m\leq r-1, \;\al \in \Pi,\; \nu \in X(T),\\
                                    
   \nu^{(r)} &\text{if } \lambda=p^{r}\nu+p^{m}w\cdot 0- p^l\alpha , \;l(w)=2,\\ 
                                    &\quad 0\leq l <  m\leq r-1, \;\al \in \Pi,\; \nu \in X(T),\\

\nu^{(r)} & \text{if } \lambda=p^{r}\nu-p^l\beta- \alpha,\\
&\quad  1\leq l \leq r-1,\; \alpha,\beta\in \Pi, \;\nu \in X(T),\\

{\nu}^{(r)}\oplus {\nu}^{(r)} & \text{if } \lambda =p^{r}\nu-p^m\beta-p^l\alpha,  \\
     &\quad 1 \leq l <m \leq r-1,\; \alpha,\beta\in \Pi,\; \nu \in X(T),\\
 
 \nu^{(r)} & \text{if } \lambda = p^r\nu -p^n \gamma - p^m \beta - p^l \alpha,\\
&\quad 0 \leq l< m < n \leq r-1,\; \alpha, \beta, \gamma \in \Pi, \;\nu \in X(T),
 \\

 {\nu}^{(r)}\oplus {\nu}^{(r)} & \text{if } \lambda =p^{r}\nu-p^l(\alpha+ \be),  \\
     &\quad  1\leq l \leq r-1, \; \alpha,\beta\in \Pi, \; \al \neq \be,\; \al + \be \notin \Phi^+,\; \nu \in X(T),\\
     
 {\nu}^{(r)} & \text{if } \la = p^r\nu - 2p^l\al,\\
	&\quad 1\leq l \leq r-1, \; \al \in \Pi, \nu \in X(T),\\

 {\nu}^{(r)} & \text{if } \lambda =p^{r}\nu-p^l(\alpha+\be),  \\
     &\quad  1\leq l \leq r-1, \; \alpha,\beta\in \Pi, \; \al + \be \in \Phi^+,\; \nu \in X(T),\\

    {\nu}^{(r)}& \text{if } \lambda =p^{r}\nu+p^ls_{\al}s_{\be} \cdot 0,  \\
     &\quad  1\leq l \leq r-1, \; \alpha,\beta\in \Pi, \; \al + \be \in \Phi^+,\; \nu \in X(T),\\

0 &\text{otherwise}.
\end{cases} 
\end{equation*} 

\end{thm}


\subsection{\bf $B$-cohomology:} From \cite[Cor. 7.2]{CPS}, we have
$\opH^3(B,\la) \cong \varprojlim \opH^3(B_r,\la)$.  If there is $\la \in X(T)$ with 
$\opH^3(B,\la) \neq 0$, then there exists $s > 0$ such that the restriction map $\opH^3(B,\la) \to \opH^3(B_r,\la)$
is non-zero for all $r \geq s$.   In particular, we must have $\opH^3(B_r,\la) \neq 0$ for all $r \geq s$.  
From Theorem~\ref{T:Brfinal}, one can readily determine those $\la$ for which $\opH^3(B,\la) \neq 0$ along with the dimensions of these groups.   These are given in Theorem~\ref{T:Bfinal}.   Since $B$ acts trivially on $\opH^{\bullet}(B,\la)$, this can then be used to compute $\opH^3(B,\la)$ for all $\la \in X(T)$
satisfying Assumption~\ref{A:char}.  This recovers and extends the work of Andersen and Rian \cite[Thm. 5.2]{AR}
who computed these groups for $p > h$.

\begin{thm}\label{T:Bfinal} Let $p$ satisfy Assumption~\ref{A:char} and $\lambda \in X(T)$. Then
\begin{equation*}
\dim \opH^{3}(B,\lambda)\cong \begin{cases}

1& \text{if } \lambda=p^l w\cdot 0, \; l(w)=3, \; l \geq 0,\\

1 &\text{if } \lambda=-p^{m}\alpha+ p^l w\cdot 0, \;l(w)=2, \; m > l \geq 0, \;\al \in \Pi,\\
                                    
   1 &\text{if } \lambda=p^{m}w\cdot 0- p^l\alpha , \;l(w)=2,\; m > l \geq 0, \;\al \in \Pi,\\ 
                                         
1 & \text{if } \lambda=-p^l\beta- \alpha,\; l \geq 1,\; \alpha,\beta\in \Pi, \\

2 & \text{if } \lambda =-p^m\beta-p^l\alpha,  \; m > l \geq 1,\; \alpha,\beta\in \Pi,\\
 
1 & \text{if } \lambda =  -p^n \gamma - p^m \beta - p^l \alpha, \; n > m > l \geq 0,
\; \alpha, \beta, \gamma \in \Pi,
 \\

 2 & \text{if } \lambda =-p^l(\alpha+ \be),  
     \;  l \geq 1, \; \alpha,\beta\in \Pi,\; \al \neq \be, \; \al + \be \notin \Phi^+,\\

 1 & \text{if } \lambda =-p^l(\alpha+\be),  
     \;  l \geq 1, \; \alpha,\beta\in \Pi, \; \al + \be \in \Phi^+,\\
     
 1 & \text{if } \la = -2p^l\al,\; l \geq 1,\; \al \in \Pi,\\

    1& \text{if } \lambda =p^ls_{\al}s_{\be} \cdot 0,  
     \;  l \geq 1, \; \alpha,\beta\in \Pi, \; \al + \be \in \Phi^+,\\

0 &\text{otherwise}.
\end{cases} 
\end{equation*} 

\end{thm}


\section{$G_r$-cohomology}


\subsection{} The computation of $B_r$-cohomology can now be used
to determine the $G_r$-cohomology of induced modules $H^0(\la)$ for some
$\la \in X(T)_+$. For $i = 1, 2$, one has
the isomorphism \cite[II.12.2]{Jan1}, \cite[Thm. 6.1]{BNP2}
$$
\opH^i(G_r,H^{0}(\la))^{(-r)} \simeq
\ind_{B}^{G}(\opH^i(B_r,\la)^{(-r)})
$$
for any $\la \in X(T)_+$. This isomorphism holds independently of
the prime and was used in \cite{BNP1,BNP2} to give explicit
descriptions of $\opH^i(G_r,H^{0}(\la))$ for all primes.
The following theorem uses the calculations done by the authors in
\cite{BNP2} and Wright \cite{W} to show that this isomorphism can be extended further
to degree three for good primes.  Recall that a prime $p$ is good if $p$ does not divide any coefficient of
a root when expressed as a sum of simple roots.

\begin{thm} Let $\la \in X(T)_+$.
\begin{itemize}
\item[(a)] Let $p$ be a good prime. Then
$$
\opH^3(G_r,H^0(\la))^{(-r)} \simeq \ind_{B}^{G}(\opH^3(B_r,\la)^{(-r)}).
$$
\item[(b)] Let $p$ satisfy Assumption~\ref{A:char}. Then  $\opH^3(G_r,H^0(\la))^{(-r)}$ has a good filtration.
\end{itemize}
\end{thm}

\begin{proof} Consider the spectral sequence
(cf. \cite[II.12.2]{Jan1})
$$
E_2^{i,j} = R^i\ind_{B}^{G}\left(\opH^j(B_r,\la)^{(-r)}\right)
\Rightarrow
\opH^{i+j}(G_r,\ind_{B}^{G}\la)^{(-r)} = \opH^{i +
j}(G_r,H^0(\la))^{(-r)}.
$$
We would like to show that
\begin{align*}
E_2^{i,0} &= R^i\ind_{B}^{G}\left(\Hom_{B_r}(k,\la)^{(-r)}\right),\\
E_2^{i,1} &= R^i\ind_{B}^{G}\left(\opH^1(B_r,\la)^{(-r)}\right), \text{ and }\\
E_2^{i,2} &= R^i\ind_{B}^{G}\left(\opH^2(B_r,\la)^{(-r)}\right)
\end{align*}
vanish for all $i > 0$. This would imply that $E^{3}\cong E_{2}^{0,3}$.

In the proof of \cite[Thm. 6.1]{BNP2}, it was shown that $E_2^{i,0} = 0$
and $E_2^{i,1} = 0$ for all $\la \in X(T)_{+}$ and all primes $p$.
For $E_2^{i,2}$, we need to consider $\opH^2(B_r,\la)^{(-r)}$ which 
was computed in \cite[Thm. 5.3, 5.7]{BNP2} for $p \geq 3$ and
\cite{W} for $p=2$.  

A careful analysis of these results shows that, as a $B$-module, and for a dominant weight $\la$, $\opH^{2}(B_{r},\la)^{(-r)}$ always has a $B$-filtration with factors of the 
form $S$ where either
\begin{itemize} 
\item[(i)] $S$ is one-dimensional of weight $\mu$ with $\langle \mu,\alpha^{\vee}\rangle\geq -1$ for all $\al \in \Pi$, or 
\item[(ii)] $S = \mathfrak u^{*}\otimes \mu$ where $\mu$ satisfies (i).
\end{itemize} 
In all the cases $R^{i}\text{ind}_{B}^{G} S=0$ for $i>0$ using 
\cite[II.5.4]{Jan1} and \cite[Thm. 2]{KLT}  (where the good prime requirement is needed).  
Therefore, $E_{2}^{i,2}=0$ for $i>0$. This proves part (a).

For a dominant weight $\la,$ Theorem~\ref{T:Brfinal} shows that $\opH^3(B_r,\la)^{(-r)}$ also has a $B$-filtration whose factors satisfy the same conditions (i) or (ii). Part (b) now follows from \cite[Thm. 7]{KLT}. 
\end{proof}


\subsection{\bf An application to $G({\mathbb F}_{q})$-cohomology:} 

Let $G({\mathbb F}_{q})$ be the finite Chevalley group obtained from $G$ by taking the ${\mathbb F}_{q}$-rational points, and $kG({\mathbb F}_{q})$ be its group algebra. 
With Theorem~\ref{T:Brfinal} we can extend the results given in \cite[Thm. 4.3.2]{BBDNPPW}. The latter theorem
required $p > h$, a condition which was needed to guarantee that $\dim\opH^3(B,\la) \leq 2$ for any weight $\la$.
From Theorem~\ref{T:Brfinal} this dimension condition holds under Assumption~\ref{A:char}.  

\begin{thm}\label{theorem:H3boundusingG}
Suppose $p$ satisfies Assumption~\ref{A:char}. Then there exists a constant $D(\Phi)$, depending on $\Phi$, such that if $r \geq D(\Phi)$ and if $q = p^r$, then, for each finite-dimensional $kG({\mathbb F}_{q})$-module $V$, one has
\[
\dim \opH^3(G({\mathbb F}_{q}),V) \leq 2 \cdot \dim V.
\]
\end{thm}

\begin{section}{Appendix}

\subsection{} For each $w \in W$, there exists a unique weight $\ga_w$ such that $w\cdot 0 + p\ga_w$ lies in the restricted region $X_1(T)$.   The lemma below provides an identification of all such weights $\ga_w$ when $\ell(w) = 3$ for $p \geq 3$. 

We first introduce some notation.  For two simple roots $\al$, $\be$, we write
$\al \sim \be$ for adjacent roots and $\al \nsim \be$ for non-ajacent roots.
Given simple roots $\al$ and $\be$, if there exists a third simple root
$\ga$ with $\al \sim \ga$ and $\ga \sim \be$ (i.e., 
we have a subgraph of the Dynkin diagram of the form $\al \leftrightarrow \ga \leftrightarrow \be$
up to a flip), we write $\al \approx \be$.  Furthermore, we write $\omega_{\al,\be}$ for $\omega_{\ga}$.

\begin{lem}\label{L:gammaw} Let $p \geq 3$. For $w = s_{\alpha_i} s_{\alpha_j}s_{\al_k} \in W$ with $\ell(w) = 3$, 
we define $\ga_w$ as follows.  Then $w \cdot 0 + p\ga_w \in X_1(T).$

\begin{itemize} 
\item[(I)] Suppose $k = i$ and $\al_i \sim \al_j$.  Then
$$ \ga_w = \omega_i + \omega_j$$
except in the following cases, where we define
$$
\ga_w = 
\begin{cases}
\omega_{n-2} + \omega_{n-1} - \omega_{n} &\text{ if $p = 3$, $\Phi$ is of type $B_n$, and } 
	w = s_{\al_{n-2}}s_{\al_{n-1}}s_{\al_{n-2}},\\
\omega_{n-1} &\text{ if $p \geq 5$, $\Phi$ is of type $B_n$, and }
	w = s_{\al_{n-1}}s_{\al_n}s_{\al_{n-1}},\\
\omega_{n-1} - \omega_{n-2} &\text{ if $p = 3$, $\Phi$ is of type $B_n$, and }
	w = s_{\al_{n-1}}s_{\al_n}s_{\al_{n-1}},\\
\omega_n &\text{ if $p \geq 5$, $\Phi$ is of type $B_n$, and }
	w = s_{\al_n}s_{\al_{n-1}}s_{\al_n},\\
2\omega_n &\text{ if $p = 3$, $\Phi$ is of type $B_n$, and }
	w = s_{\al_n}s_{\al_{n-1}}s_{\al_n},\\
\omega_{n-1} &\text{ if $p \geq 5$, $\Phi$ is of type $C_n$, and }
	w = s_{\al_{n-1}}s_{\al_n}s_{\al_{n-1}},\\
2\omega_{n-1} - \omega_{n-2} &\text{ if $p = 3$, $\Phi$ is of type $C_n$, and }
	w = s_{\al_{n-1}}s_{\al_n}s_{\al_{n-1}},\\
\omega_n &\text{ if $p \geq 5$, $\Phi$ is of type $C_n$, and }
	w = s_{\al_n}s_{\al_{n-1}}s_{\al_n},\\
\omega_n - \omega_{n-2} &\text{ if $p = 3$, $\Phi$ is of type $C_n$, and }
	w = s_{\al_n}s_{\al_{n-1}}s_{\al_n},\\
\omega_1 + \omega_2 - \omega_3 &\text{ if $p = 3$, $\Phi$ is of type $F_4$, and }
	w = s_{\al_1}s_{\al_2}s_{\al_1},\\
\omega_2  &\text{ if $p \geq 5$, $\Phi$ is of type $F_4$, and }
	w = s_{\al_2}s_{\al_3}s_{\al_2},\\
\omega_2 - \omega_1 - \omega_4 &\text{ if $p = 3$, $\Phi$ is of type $F_4$, and }
	w = s_{\al_2}s_{\al_3}s_{\al_2},\\
\omega_3 + \omega_2 - \omega_4 &\text{ if $p = 3$, $\Phi$ is of type $F_4$, and }
	w = s_{\al_3}s_{\al_2}s_{\al_3},\\	
\omega_1 &\text{ if $p \geq 7$, $\Phi$ is of type $G_2$, and }
	w = s_{\al_1}s_{\al_2}s_{\al_1},\\
2\omega_1 &\text{ if $p = 3, 5$, $\Phi$ is of type $G_2$, and }
	w = s_{\al_1}s_{\al_2}s_{\al_1},\\
\omega_2 &\text{ if $p \geq 5$, $\Phi$ is of type $G_2$, and }
	w = s_{\al_2}s_{\al_1}s_{\al_2},\\
2\omega_2 - \omega_1 &\text{ if $p = 3$, $\Phi$ is of type $G_2$, and }
	w = s_{\al_2}s_{\al_1}s_{\al_2}.
\end{cases}
$$

\medskip\noindent
Suppose from now on that $\al_i$, $\al_j$, and $\al_k$ are distinct.

\medskip
\item[(II)] Suppose none of the simple roots $\al_i$, $\al_j$, and $\al_k$ are adjacent to each other.
Then
$$\ga_w = \omega_i + \omega_j + \omega_k
$$
except in the following $p=3$ cases, where we define
$$\ga_w = 
\begin{cases}
\omega_{n}  + \omega_{n-2} + \omega_{k} + \omega_{n-1} & \text{ if $\Phi$ is of type $C_n$, and } 
	w= s_{\alpha_{n}} s_{\alpha_ {n-2}}s_{\al_{k}}\\
	&\quad \text{ with } k \leq n - 4,\\
\omega_{n} + \omega_{n-1} + \omega_{n-3} - \omega_{n-2}  & \text{ if $\Phi$  of type $D_n$, and } 
	w= s_{\al_{n}}s_{\alpha_{n-1}} s_{\alpha_ {n-3}},\\
\omega_2 + \omega_3 + \omega_5 - \omega_4  &\text{ if $\Phi$ is of type $E_n$ and } 
	w= s_{\alpha_{2}} s_{\alpha_ {3}}s_{\al_5}.
\end{cases}
$$

\medskip
\item[(III)] Suppose that precisely one pair of simple roots from $\al_i$, $\al_j$, and $\al_k$ are 
adjacent.  By rewriting $w$ if necessary, we may assume that $\al_i \sim \al_j$, $\al_i \nsim \al_k$,
and $\al_j \nsim \al_k$.  Then
$$\ga_w = \omega_i + \omega_k$$
unless $p = 3$ and $\al_i \approx \al_k$, in which case
$$\ga_w = \omega_i + \omega_k - \omega_{i,k}$$
except in the following cases, where we define
$$
\ga_w = 
\begin{cases}
2\omega_n + \omega_k &\text{ if $p = 3$, $\Phi$ is of type $B_n$, and }
	w = s_{\al_n}s_{\al_{n-1}}s_{\al_k}\\
	&\quad \text{ with } k \leq n - 3,\\
2\omega_{n-1} + \omega_k - \omega_{n-2} &\text{ if $p = 3$, $\Phi$ is of type $C_n$, and }
	w = s_{\al_{n-1}}s_{\al_n}s_{\al_{k}} \\
	&\quad \text{ with } k \leq n - 3,\\
\omega_1 + \omega_4 - \omega_3 &\text{ if $p = 3$, $\Phi$ is of type $F_4$, and }
	w = s_{\al_1}s_{\al_2}s_{\al_4},\\
\omega_2 + \omega_4 - \omega_3 &\text{ if $p = 5$, $\Phi$ is of type $F_4$, and }
	w = s_{\al_2}s_{\al_1}s_{\al_4}.
\end{cases}
$$

\medskip
\item[(IV)] Suppose $\al_i \sim \al_j$ and $\al_j \sim \al_k$ (i.e., there is a subgraph 
of the form $\al_i \leftrightarrow \al_j \leftrightarrow \al_k$).
Then
$$\ga_w = 
\begin{cases}
\omega_i &\text{ if } p \geq 5,\\
2\omega_i - \sum_{\al_{\ell} \sim \al_i, \ell \neq j} \omega_{\ell} &\text{ if } p = 3,
\end{cases}
$$
except in the following cases, where we define
$$
\ga_w = 
\begin{cases}
\omega_{n-1} - \omega_n &\text{ if } p \geq 5, \Phi \text{ is of type } B_n, \text{ and }
	w = s_{\al_{n-1}}s_{\al_{n-2}}s_{\al_{n-3}},\\
2\omega_{n-1} - 2\omega_n &\text{ if } p = 3, \Phi \text{ is of type } B_n, \text{ and }
	w = s_{\al_{n-1}}s_{\al_{n-2}}s_{\al_{n-3}},\\
2\omega_{n} &\text{ if } p \geq 5, \Phi \text{ is of type } B_n, \text{ and }
	w = s_{\al_{n}}s_{\al_{n-1}}s_{\al_{n-2}},\\
2\omega_{n} - \omega_{n-1} &\text{ if } p = 3, \Phi \text{ is of type } C_n, \text{ and }
	w = s_{\al_{n}}s_{\al_{n-1}}s_{\al_{n-2}},\\
2\omega_2 - \omega_1 - \omega_3 &\text{ if } p = 3, \Phi \text{ is of type } F_4, \text{ and }
	w = s_{\al_2}s_{\al_3}s_{\al_4}.
\end{cases}
$$

\medskip
\item[(V)] Suppose $\al_i \sim \al_j$ and $\al_i \sim \al_k$ (i.e., there is a subgraph 
of the form $\al_j \leftrightarrow \al_i \leftrightarrow \al_k$).
Then
$$\ga_w = 
\begin{cases}
\omega_i &\text{ if } p \geq 5,\\
2\omega_i - \sum_{\al_{\ell} \sim \al_i, \ell \neq j,k} \omega_{\ell} &\text{ if } p = 3,
\end{cases}
$$
except in the following cases, where we define
$$
\ga_w = 
\begin{cases}
\omega_{n-1} - \omega_n &\text{ if } p = 5, \Phi \text{ is of type } B_n, \text{ and }
	w = s_{\al_{n-1}}s_{\al_{n-2}}s_{\al_n},\\
2\omega_{n-1} - 2\omega_n &\text{ if } p = 3, \Phi \text{ is of type } B_n, \text{ and }
	w = s_{\al_{n-1}}s_{\al_{n-2}}s_{\al_n},\\
2\omega_{n-1} &\text{ if } p = 5, \Phi \text{ is of type } C_n, \text{ and }
	w = s_{\al_{n-1}}s_{\al_{n-2}}s_{\al_n},\\
2\omega_{n-1} - \omega_n &\text{ if } p = 3, \Phi \text{ is of type } C_n, \text{ and }
	w = s_{\al_{n-1}}s_{\al_{n-2}}s_{\al_n},\\
\omega_{2} - \omega_3 &\text{ if } p = 5, \Phi \text{ is of type } F_4, \text{ and }
	w = s_{\al_{2}}s_{\al_{1}}s_{\al_3},\\
2\omega_{2} - 2\omega_3 &\text{ if } p = 3, \Phi \text{ is of type } F_4, \text{ and }
	w = s_{\al_{2}}s_{\al_{1}}s_{\al_3},\\
2\omega_3 &\text{ if } p = 5, \Phi \text{ is of type } F_4, \text{ and }
	w = s_{\al_{3}}s_{\al_{4}}s_{\al_2},\\
2\omega_3 - \omega_2 &\text{ if } p = 3, \Phi \text{ is of type } F_4, \text{ and }
	w = s_{\al_{3}}s_{\al_{4}}s_{\al_2}.
\end{cases}
$$

\medskip
\item[(VI)] Suppose $\al_i \sim \al_k$ and $\al_j \sim \al_k$ (i.e., there is a subgraph 
of the form $\al_i \leftrightarrow \al_k \leftrightarrow \al_j$).
Then
$$\ga_w = \omega_i + \omega_j
$$
except in the following cases when $p = 3$, where we define
$$
\ga_w = 
\begin{cases}
\omega_{n-1} + \omega_{n-3} - \omega_{n} &\text{ if } \Phi \text{ is of type } B_n \text{ and }
	w = s_{\al_{n-1}}s_{\al_{n-3}}s_{\al_{n-2}},\\
2\omega_{n} + \omega_{n-2} - \omega_{n-1} &\text{ if } \Phi \text{ is of type } B_n \text{ and }
	w = s_{\al_{n}}s_{\al_{n-2}}s_{\al_{n-1}},\\
\omega_{n} + \omega_{n-2} - \omega_{n-1} &\text{ if } \Phi \text{ is of type } C_n \text{ and }
	w = s_{\al_{n}}s_{\al_{n-2}}s_{\al_{n-1}},\\
\omega_{1} + \omega_{3} - \omega_{2} - \omega_{4} &\text{ if } \Phi \text{ is of type } F_4 \text{ and }
	w = s_{\al_{1}}s_{\al_{3}}s_{\al_{2}},\\
\omega_{2} + \omega_{4} - \omega_{3} &\text{ if } \Phi \text{ is of type } F_4 \text{ and }
	w = s_{\al_{2}}s_{\al_{4}}s_{\al_{3}}.\\
\end{cases}
$$
\end{itemize} 
\end{lem}
\end{section}



\end{document}